\documentclass[a4paper,11pt]{amsart}
\usepackage{amssymb}
\usepackage[OT1]{fontenc}
\usepackage[latin1]{inputenc}
\usepackage[dvips]{hyperref}
\usepackage[curve]{xypic}
\usepackage{enumerate}
\usepackage[active]{srcltx}
\usepackage{ae,aecompl}

\numberwithin{equation}{section}

\def\tot{\tau^\omega}
\def\ci#1{\ensuremath{{\mathcal {#1}}}}
\def\z#1{\ensuremath{{{#1}}^{\ZZ}}}
\def\pv#1{\ensuremath{{\mathsf{#1}}}}
\def\Om#1#2{\ensuremath{\overline\Omega_{#1}{\pv{#2}}}}
\def\Im#1{\ensuremath{\mathop{\rm Im}{#1}}}
\def\KerS#1{\ensuremath{\mathop{\rm Ker}{#1}}}
\def\End#1{\ensuremath{\mathop{\rm End}{#1}}}
\def\ZZ{\ensuremath{\mathbb{Z}}}

\newtheorem{teor}{Theorem}[section]
\newtheorem{prop}[teor]{Proposition}
\newtheorem{lema}[teor]{Lemma}

\newtheorem{cor}[teor]{Corollary}
\newtheorem{remark}[teor]{Remark}
\newtheorem{prob}[teor]{Problem}

\theoremstyle{definition}

\CompileMatrices

\begin{document}

\title
{
  Presentations of Sch\"utzenberger groups of minimal subshifts
}

\thanks{
  Research funded by the European Regional Development Fund,
  through the programme COMPETE, by the Portuguese Government through
  Centro de Matem\'atica da Universidade do Porto, Centre for
  Mathematics of the University of Coimbra, and FCT -- Funda\c{c}\~ao
  para a Ci\^encia e a Tecnologia, under the projects
  PEst-C/MAT/UI0144/2011 and PEst-C/MAT/UI0324/2011, and by the FCT
  project PTDC/MAT/65481/2006, within the framework of the programmes
  COMPETE and FEDER
}

\author{Jorge Almeida}
\address{CMUP, Departamento de Matem\'atica,
     Faculdade de Ci\^encias, Universidade do Porto, 
 Rua do Campo Alegre 687, 4169-007 Porto, Portugal.}
\email{jalmeida@fc.up.pt}

\author{Alfredo Costa}
\address{CMUC, Department of Mathematics, University of Coimbra,
  3001-454 Coimbra, Portugal.}
\email{amgc@mat.uc.pt}

\begin{abstract}
  In previous work, the first author established a natural bijection
  between minimal subshifts and maximal regular $\ci J$-classes of
  free profinite semigroups. In this paper, the Sch\"utzenberger
  groups of such $\ci J$-classes are investigated, in particular in
  respect to a conjecture proposed by the first author concerning
  their profinite presentation. The conjecture is established for all
  non-periodic minimal subshifts associated with substitutions. It
  entails that it is decidable whether a finite group is a quotient of
  such a profinite group. As a further application, the Sch\"utzenberger
  group of the $\ci J$-class corresponding to the Prouhet-Thue-Morse
  subshift is shown to admit a somewhat simpler presentation, from
  which it follows that it has rank three, and that it is non-free
  relatively to any pseudovariety of groups.
\end{abstract}


\makeatletter
\@namedef{subjclassname@2010}{%
  \textup{2010} Mathematics Subject Classification}
\makeatother

\maketitle

\section{Introduction}

In recent years, several results on closed subgroups of free profinite
semigroups have appeared in the literature~\cite{Almeida:2005c,
  Almeida:2003c, Almeida&Volkov:2006, Rhodes&Steinberg:2008,
  Steinberg:2009,ACosta&Steinberg:2011}. The first author explored a
link between symbolic dynamics and free profinite semigroups that
allowed him to show, for several classes of maximal subgroups of free
profinite semigroups, all associated with minimal
subshifts~\cite{Almeida:2005c, Almeida:2003c}, that they are free
profinite groups. Rhodes and Steinberg~\cite{Rhodes&Steinberg:2008}
proved that the closed subgroups of free profinite semigroups are
precisely the projective profinite groups. Without using ideas from
symbolic dynamics, Steinberg proved that the Sch\"utzenberger group of
the minimal ideal of the free profinite semigroup over a finite
alphabet with at least two letters is a free profinite group with
infinite countable rank~\cite{Steinberg:2009}. The same result holds
for the Sch\"utzenberger group of the regular $\ci J$-class associated
to a non-periodic irreducible sofic
subshift~\cite{ACosta&Steinberg:2011}; the proof is based on the
techniques of~\cite{Steinberg:2009} and on the conjugacy invariance of
the group for arbitrary subshifts~\cite{ACosta:2006}.

In this paper, we investigate the minimal subshift associated with the
iteration of a substitution $\varphi$ over a finite alphabet $A$ and
the Sch\"utzenberger group $G(\varphi)$ of the corresponding $\ci
J$-class, $J(\varphi)$, of the free profinite semigroup on~$A$. A
minimal subshift can be naturally associated with the substitution
$\varphi$ if and only if $\varphi$ is \emph{weakly primitive}
\cite[Theorem~3.7]{Almeida:2005c}. Since weakly primitive
substitutions are \emph{primitive} on the subalphabet consisting of
the letters that do not eventually disappear under iteration of the
substitution, we will stick in this paper to the more familiar setting
of primitive substitutions \cite{Fogg:2002}.

A primitive substitution always admits a so-called \emph{connection},
which is a special two-letter block $ba$ of the subshift. Provided
$\varphi$~is an encoding of bounded delay, from the set $X$ of return
words for $ba$, which constitute a finite set, it is shown
in~\cite{Almeida:2005c} that one can then obtain a generating set for
a certain maximal subgroup~$H$ of~$J(\varphi)$ by cancelling the
prefix $b$, adding the same letter as a suffix, and applying the
idempotent (profinite) iterate~$\varphi^\omega$. In a lecture given at
the \emph{Fields Workshop on Profinite Groups and Applications}
(Carleton University, August 2005), the first author proposed, as a
problem, a natural profinite presentation for $G(\varphi)$, namely
\begin{equation}
  \label{eq:presentation}
  \langle X\mid\Phi^\omega(x)=x\ (x\in X)\rangle,
\end{equation}
where $\Phi$ is a continuous endomorphism of the free profinite group
on a suitable finite alphabet~$X$ that encodes the action of a finite
power of $\varphi$ which acts on the semigroup freely generated
by~$X$.

By a result of Lubotzky and Kov\'acs~\cite{Lubotzky:2001}, every
finitely generated projective profinite group has a finite
presentation as a profinite group, and indeed a presentation of the
form~\eqref{eq:presentation} for some continuous endomorphism $\Phi$
of the profinite group freely generated by~$X$. Hence, by the
previously mentioned result of Rhodes and Steinberg, every finitely
generated closed subgroup of a free profinite semigroup has such a
presentation. But, to be able to use a presentation of the
form~\eqref{eq:presentation}, for instance to determine whether a
given finite group is a (continuous) homomorphic image of the
profinite group so presented, one needs to be able to verify the
relations in a finite group, which imposes some computability
requirements on~$\Phi$. The problem proposed by the first author in
2005 already addressed this concern, proposing a suitable choice
for~$\Phi$.

In this paper, we establish the conjecture in full generality, that is
without any further restrictions on the (weakly) primitive
substitution $\varphi$ other than being non-periodic
(Theorem~\ref{t:conjecture}, which is our main theorem), thereby
showing that it entails the decidability of whether a finite group is
a continuous homomorphic image of~$G(\varphi)$
(Corollary~\ref{c:special-presentation->decidable}).\footnote{It is
  worth noting that it is decidable whether a given primitive
  substitution generates a periodic subshift
  \cite{Pansiot:1986,Harju&Linna:1986}.} The proof of the conjecture
depends on a key result from symbolic dynamics due to Moss\'e
\cite{Mosse:1992,Mosse:1996} (see \cite[Subsection~7.2.1]{Fogg:2002}
for its significance and history). Its need had been previously avoided
in~\cite{Almeida:2005c} using the bounded delay encoding condition,
which is fulfilled in the case of substitutions that induce
automorphisms of the free group.

The case of a \emph{proper} substitution, such that the images of all
letters start with the same letter and end with the same letter, has
played a special role both in symbolic dynamics
\cite{Durand&Host&Skau:1999} and in the connections with free
profinite semigroups \cite{Almeida&Volkov:2006,Almeida:2005c}. The
former reference shows that every subshift generated by a primitive
substitution is conjugate to a subshift generated by a proper primitive
substitution, which can be effectively computed. Since conjugate
minimal subshifts have isomorphic Sch\"utzenberger groups
\cite{ACosta:2006}, it is worth considering the special case of
subshifts generated by proper primitive substitutions, whose
Sch\"utzenberger group we show to admit the
presentation~\eqref{eq:presentation} with $X$ the original alphabet
and $\Phi$ the original substitution, provided the subshift is
non-periodic (Theorem~\ref{t:presentation-AV-case}). This gives an
alternative approach for the main theorem and its decidability
consequences.

The Prouhet-Thue-Morse infinite word and the corresponding subshift
are among the most studied in the literature~\cite{Fogg:2002}. They
are generated by the substitution $\tau(a)=ab$, $\tau(b)=ba$. From the
main theorem, we deduce that the profinite group $G(\tau)$ admits a
related profinite presentation with three generators and three
relations (Theorem~\ref{t:presentation-PTM}).
We deduce that $G(\tau)$ cannot be
relatively free with respect to any pseudovariety of groups
(Theorem~\ref{t:not-relatively-free-profinite}). This answers in a
very strong sense the question raised by the first author as to
whether this profinite group is free~\cite{Almeida:2005c}. In the same
paper there is already an argument to reduce the proof of this fact to
showing that the Sch\"utzenberger group $G(\tau)$ has rank three. From
the same simpler presentation, we do prove that this group has rank
three (Theorem~\ref{t:rank}).

We also consider the only other type of example in the literature of a
non-free Sch\"utzenberger group $G(\varphi)$ of a subshift defined by
a substitution, illustrated by the substitution $\varphi(a)=ab$,
$\varphi(b)=a^3b$ \cite[Example~7.2]{Almeida:2005c}, which is proper.
For this group, again we prove that it is not free relatively to any
pseudovariety of groups (Theorem~\ref{t:eg:1}).

The paper is organized as follows.
Section~\ref{sec:presentations-general} discusses presentations of
profinite semigroups. Section~\ref{sec:decidability} shows how certain
presentations can be used to obtain decidability results, which is our
main motivation for considering profinite presentations.
Section~\ref{sec:preliminaries} introduces the necessary background
and terminology on symbolic dynamics. The result of B. Moss\'e and its
consequence that a power of any non-periodic primitive substitution
$\varphi$ induces an automorphism of a suitable maximal subgroup
of~$J(\varphi)$ (Theorem~\ref{t:in-Im-phi}) are presented in
Section~\ref{sec:max-subgps-as-retracts}. Section~\ref{sec:core}
contains the main theorem and its version for proper primitive
substitutions, as well as the connections between the two.
Section~\ref{sec:applications} is dedicated to applications of the
main theorems and Section~\ref{sec:problems} concludes with some
open problems suggested by this work.

We indicate~\cite{Almeida:2003c,Rhodes&Steinberg:2009qt} as
supporting references on pseudovarieties and free profinite semigroups,
and~\cite{Lind&Marcus:1996,Fogg:2002} for symbolic dynamics.

\section{Presentations of pro-\texorpdfstring{\pv V}{V} semigroups}
\label{sec:presentations-general}

For a homomorphism $\psi:S\to U$ between semigroups, we denote by
$\KerS\psi$ the set of all pairs $(s_1,s_2)$ of elements of~$S$ such
that $\psi(s_1)=\psi(s_2)$.

It can be easily checked that an equivalence relation on a compact
space is open (respectively, closed, clopen) if so are its classes. In
particular, an equivalence relation on such a space is open if and
only if it is clopen. A congruence on a profinite semigroup $S$ is
said to be \emph{admissible} if it is the intersection of open
congruences. In other words, a congruence $\rho$ is admissible if and
only if it is closed and the quotient $S/\rho$~is profinite. Thus, the
admissible congruences are the kernels of continuous homomorphisms
into profinite semigroups. Since the intersection of admissible
congruences is admissible, for every relation $R\subseteq S\times S$
there is a smallest admissible congruence containing $R$, which we
call the \emph{admissible congruence generated} by~$R$. In the case of
a profinite group, it turns out that a congruence is admissible if and
only if it is closed
\cite[Proposition~2.2.1(a)]{Ribes&Zalesskii:2000}. See
\cite[Section~3.1]{Rhodes&Steinberg:2009qt} for further details,
although we prefer not to call \emph{profinite} an admissible
congruence on a profinite semigroup $S$ since every closed congruence
is a profinite subsemigroup of the product $S\times S$, but not every
closed congruence is admissible.

Throughout this section, we let \pv V be a pseudovariety of
semigroups. Consider a set $X$ and a binary relation $R$ on the
pro-\pv V semigroup \Om XV freely generated by~$X$
\cite{Almeida:2003c}. The quotient of \Om XV by the admissible
congruence generated by~$R$ is a pro-\pv V semigroup
\cite[Proposition~3.7]{Almeida:2003c} which is said to admit the
\emph{\pv V-presentation} $\langle X\mid R\rangle_\pv V$. In this
paper, we are interested in the cases where \pv V is either~\pv G, the
pseudovariety of all finite groups, or~\pv S, the pseudovariety of all
finite semigroups, the latter serving sometimes as a convenient way to
deal with the former.

We recall that the monoid \End S of continuous endomorphisms of a
finitely generated profinite semigroup $S$ is profinite for the
pointwise convergence topology, which coincides with the compact-open
topology~\cite[Proposition~1]{Hunter:1983}. For this reason, for the
remainder of the paper we only consider finite generating sets. Thus,
for $\varphi$ in the profinite monoid $\End{\Om XS}$, we may consider
the idempotent continuous endomorphism~$\varphi^\omega$.

Consider a pro-\pv V semigroup $T$ and an onto continuous homomorphism
$\pi$ from $\Om XV$ onto $T$, where $X$ is an arbitrary set. Let
$\varphi$ be a continuous endomorphism of $T$. By the universal
property of $\Om XV$, there is at least one continuous endomorphism
$\Phi$ of~\Om XV such that Diagram~\eqref{eq:a-com-diagram} commutes.
Call such an endomorphism a \emph{lifting of $\varphi$ via $\pi$}.
\begin{equation}\label{eq:a-com-diagram}
  \begin{split}
    \xymatrix{
      \Om XV\ar[r]^{\Phi}\ar[d]_\pi&\Om XV\ar[d]^\pi\\
      {T}\ar[r]^\varphi&{T}
    }
  \end{split}
\end{equation}

\begin{remark}
  \label{r:2}
  If $\varphi$ is an automorphism of $T$ then
  $\pi\circ\Phi^\omega=\pi$.
\end{remark}

\begin{proof}
  The facts that Diagram~\eqref{eq:a-com-diagram} commutes and
  $\pi$~is continuous entail the equality
  $\pi\circ{\Phi}^{\omega}=\varphi^\omega\circ\pi$. On the other hand,
  $\varphi^\omega$ is the identity on $T$ because $\varphi$ is an
  automorphism of $T$.
\end{proof}

Suppose now that $\varphi$ is an automorphism of the pro-\pv V
semigroup~$T$. Put
$$R=\{(\Phi^\omega(x),x):x\in X\}$$
and let $\rho$ be the admissible congruence on~\Om XV generated
by~$R$. From Remark~\ref{r:2} it follows that $R\subseteq\KerS\pi$,
which yields $\rho\subseteq \KerS\pi$. If $\rho=\KerS\pi$, then
\begin{equation}
  \label{eq:V-presentation}
  \langle X\mid \Phi^\omega(x)=x\ (x\in X)\rangle_\pv V
\end{equation}
is a presentation of~$T$. Note also that
$u\mathrel{\rho}\Phi^\omega(u)$ for every $u\in\Om XS$ since $\rho$ is
a closed congruence containing~$R$. It follows that
$\KerS\Phi^\omega\subseteq\rho$. Conversely, we have $R\subseteq
\KerS\Phi^\omega$ since $\Phi^\omega$ is idempotent, which entails
$\rho\subseteq\KerS\Phi^\omega$. We have thus shown that
$\rho=\KerS\Phi^\omega$.

\begin{lema}\label{l:equivalent-conditions-s}
  Let $T$ be a pro-\pv V semigroup and suppose that there is a
  commutative diagram \eqref{eq:a-com-diagram} of continuous
  homomorphisms, where $\pi$ is onto and $\varphi$ is an automorphism
  of~$T$. If \/ $\KerS\pi\subseteq\KerS\Phi^\omega$, then $T$ admits
  the presentation $\langle X\mid R\rangle_\pv V$.
\end{lema}

\begin{proof}
  By Remark~\ref{r:2}, we have $\KerS\pi\supseteq\KerS\Phi^\omega$.
  Hence, if $\KerS\pi\subseteq\KerS\Phi^\omega$, then $T\simeq\Om
  XV/\KerS\Phi^\omega=\Om XS/\rho=\langle X\mid R\rangle_\pv V$.
\end{proof}

The group analogue of Lemma~\ref{l:equivalent-conditions-s} involving
the group kernel, which is just a translation in the language of
profinite group theory of the lemma, also follows from a result of
Lubotzky~\cite[Proposition 1.1]{Lubotzky:2001}, who presents a proof
attributed to L. Kov\'acs. The same proof can also be found in the
second edition of~\cite{Ribes&Zalesskii:2000}, namely by combining
Lemma C.1.5 and Example C.1.6.

Let \pv W be a subpseudovariety of~\pv V. For a pro-\pv W semigroup,
there is a simple relationship between \pv V-presentations and \pv
W-presentations. If $\lambda:\Om XV\to S$ is a continuous homomorphism
onto a pro-\pv W semigroup, then $\lambda=\lambda'\circ q$, where
$q:\Om XV\to\Om XW$ is the canonical homomorphism and $\lambda':\Om
XW\to S$ is a continuous homomorphism. Let $u,v\in\Om XV$. It is
routine to check that if $\KerS\lambda$ is the admissible congruence
on~\Om XV generated by $R\subseteq\Om XV\times\Om XV$, then
$\KerS\lambda'$ is the admissible congruence on~\Om XW generated by
$(q\times q)(R)$. We thus have the following simple observation, which
we record here for later reference.

\begin{lema}
  \label{l:folklore}
  Let \pv W be a subpseudovariety of~\pv V. If the pro-\pv W semigroup
  $S$ admits the presentation $\langle X\mid R\rangle_\pv V$, then it
  also admits the presentation $\langle X\mid (q\times
  q)(R)\rangle_\pv W$.\qed
\end{lema}

We say that a pro-\pv V semigroup $S$ is \emph{\pv V-projective} if,
whenever $T$ and $U$ are pro-\pv V semigroups and $f:S\to T$ and
$g:U\to T$ are continuous homomorphisms with $g$ onto, there is
some continuous homomorphism $f':S\to U$ such that
the following diagram commutes:
$$
\xymatrix{
  &
  S
  \ar[d]^f
  \ar@{-->}[ld]_{f'}
  \\
  U
  \ar[r]_g
  &
  \,T.
}
$$

\begin{prop}
  \label{p:retracts}
  The following are equivalent for a pro-\pv V semigroup~$S$ and a
  finite set~$X$:
  \begin{enumerate}[(1)]
  \item\label{item:retracts-1} $S$ admits a presentation of the
    form~(\ref{eq:V-presentation}) for some continuous
    endomorphism $\Phi$ of\/~\Om XV;
  \item\label{item:retracts-2} $S$ is \pv V-projective and $X$-generated;
  \item\label{item:retracts-3} $S$ is a retract of\/~\Om XV.
  \end{enumerate}
\end{prop}

\begin{proof}
  (\ref{item:retracts-1})\,$\Rightarrow$\,(\ref{item:retracts-3}) Let
  $\Phi$ be a continuous endomorphism of~\Om XV and denote by $T$ the
  image of the retraction $\Phi^\omega$. It suffices to establish that
  $T$ admits the presentation~(\ref{eq:V-presentation}). For this
  purpose, we apply the general setting of this section to the
  following commutative diagram:
  $$
  \xymatrix{
    \Om XV
    \ar[d]_{\Phi^\omega}
    \ar[r]^{\Phi^\omega}
    &
    \Om XV
    \ar[d]^{\Phi^\omega}
    \\
    T
    \ar[r]^{\mathrm{id}}
    &
    T
  }
  $$
  From Lemma~\ref{l:equivalent-conditions-s} we deduce that indeed $T$
  admits the presentation~(\ref{eq:V-presentation}).
  
  (\ref{item:retracts-2})\,$\Rightarrow$\,(\ref{item:retracts-1})
  Let $\pi:\Om XV\to S$ be an onto continuous homomorphism. Since $S$
  is \pv V-projective, there is a continuous homomorphism
  $\gamma:S\to\Om XV$ such that $\pi\circ\gamma$ is the identity
  on~$S$. Consider the diagram
  $$
  \xymatrix{
    \Om XV
    \ar[r]^{\gamma\circ\pi}
    \ar[d]_\pi
    &
    \Om XV
    \ar[d]^\pi
    \\
    S
    \ar[r]_{\mathrm{id}}
    \ar[ru]_\gamma
    &
    \,S.
  }
  $$
  Since $\gamma\circ\pi$ is idempotent and
  $\KerS\pi\subseteq\KerS{(\gamma\circ\pi)}$,
  Lemma~\ref{l:equivalent-conditions-s} yields that $S$ admits the
  presentation $\langle X\mid (\gamma\circ\pi)^\omega(x)=x\ (x\in
  X)\rangle_\pv V$.

  (\ref{item:retracts-3})\,$\Rightarrow$\,(\ref{item:retracts-2})
  Suppose that the continuous homomorphism $r:\Om XV\to S$ is a
  retraction. Let $T$ and $U$ be pro-\pv V semigroups and let $f:S\to
  T$ and $g:U\to T$ be continuous homomorphisms with $g$ onto. Denote
  by $ i$ the inclusion mapping $S\to\Om XV$. Then we have the
  following diagram
  $$
  \xymatrix@C=5em{
    &
    \Om XV
    \ar[ldd]_ h
    \ar@<-0.5ex>[d]_ r
    \ar@/^.6cm/[dd]^{f\circ r}
    \\
    &
    S
    \ar@<-0.5ex>[u]_ i
    \ar[ld]^{ h\circ i}
    \ar[d]_f
    \\
    U
    \ar[r]_g
    &
    \,T,
  }
  $$
  where the existence of the continuous homomorphism $ h$ such
  that the outer triangle commutes follows from the universal
  property of the free pro-\pv V semigroup \Om XV. Since
  $ r\circ i$ is the identity on~$S$, we deduce that
  $g\circ h\circ i=f\circ r\circ i=f$, and
  so we may take $f'= h\circ i$.
\end{proof}

Combining Proposition~\ref{p:retracts} with the fact that closed
subgroups of a free profinite semigroup are \pv G-projective
\cite{Rhodes&Steinberg:2008}, we obtain the following result.

\begin{cor}
  \label{c:there-is-some-presentation-of-that-kind}
  Every finitely generated closed subgroup of a free profinite
  semigroup admits a presentation of the form
  \begin{equation}
    \label{eq:G-presentation}
    \langle X\mid \Phi^\omega(x)=x\ (x\in X)\rangle_\pv G
   \end{equation}
  for some continuous endomorphism $\Phi$ of~\Om
  XG.\qed
\end{cor}

The next result provides a method to drop relations in such
presentations corresponding to superfluous generators.

For a profinite semigroup $S$ and a subset $X$, the notation
$\overline{\langle X\rangle}$ stands for the closed subsemigroup of
$S$ generated by $X$.

\begin{prop}
  \label{p:drop-relation}
  Let $\Phi$ be a continuous endomorphism of~\Om XV, $x_0$ an element
  of the finite set $X$, and $Y=X\setminus\{x_0\}$. Suppose that
  $w\in\overline{\langle Y\rangle}$ is such that
  $\Phi^\omega(x_0)=\Phi^\omega(w)$ and let $r$ be the unique
  continuous endomorphism of~\Om XV that fixes each $y\in Y$ and maps
  $x_0$ to~$w$. Then the pro-\pv V semigroup presented by
  \begin{equation}
    \label{eq:drop-relation-1}
    \langle X\mid w=x_0,\, \Phi^\omega(x)=x\ (x\in X)\rangle_\pv V
  \end{equation}
  also admits
  the presentation
  \begin{equation}
    \label{eq:drop-relation-2}
    \langle Y\mid \Psi^\omega(y)=y\ (y\in Y)\rangle_\pv V,
  \end{equation}
  where $\Psi=r\circ\Phi$.
\end{prop}

\begin{proof}
  Note that we may add in the presentation~\eqref{eq:drop-relation-2}
  the generator $x_0$ and the relation $w=x_0$ without changing the
  pro-\pv V semigroup thus presented.
  
  Let $\theta$ be the admissible congruence on~\Om XV generated by the
  relation $w=x_0$. Since $r(u)\mathrel{\theta}u$ for every $u\in\Om
  XV$, we conclude that $\Psi(v)=r(\Phi(v))\mathrel{\theta}\Phi(v)$
  whenever $v\in\Om XV$.

  Let $\rho$ and $\sigma$ be the admissible congruences on~\Om XV
  generated by the relation $w=x_0$ together with, respectively, the
  relations $\Phi^\omega(x)=x$ ($x\in X$) and $\Psi^\omega(y)=y$
  ($y\in Y$). To complete the proof, it suffices to show that
  $\rho=\sigma$. For this purpose, in view of the preceding paragraph,
  it remains to show that $\Phi^\omega(x_0)\mathrel{\sigma}x_0$.
  Indeed,  we have
  $$\Phi^\omega(x_0)
  =\Phi^\omega(w)
  \mathrel{\theta}\Psi^\omega(w)
  \mathrel{\sigma}w
  \mathrel{\sigma}x_0,
  $$
  which gives the desired relation since $\theta\subseteq\sigma$.
\end{proof}

Note that, with the same proof, we could relax the hypothesis
$\Phi^\omega(x_0)=\Phi^\omega(w)$ to the relation
$\Phi^\omega(x_0)\mathrel{\theta}\Phi^\omega(w)$.

\section{Decidability}
\label{sec:decidability}

For a set $X$, denote by $\ci T(X)$ the semigroup of all full
transformations of~$X$. The following lemma will be useful. As has
been observed by the referee, it can be seen as an application of
Yoneda's Lemma but we prefer to give an elementary proof.

\begin{lema}
  \label{l:technical}
  Let $A$ be a finite set, \pv V be a pseudovariety of semigroups,
  $\varphi\in\End{\Om AV}$, and $S$ a semigroup from~\pv V. Consider
  the transformation $\varphi_S\in\ci T(S^A)$ defined by
  $\varphi_S(f)=\hat f\circ\varphi|_A$, where $\hat f$ is the unique
  extension of $f\in S^A$ to a continuous homomorphism $\Om AV\to S$.
  Then the correspondence
  \begin{align*}
    \End{\Om AV}&\to\ci T(S^A)\\
    \varphi&\mapsto\varphi_S
  \end{align*}
  is a continuous anti-homomorphism. In particular, we have
  $(\varphi^\omega)_S=(\varphi_S)^\omega$.
\end{lema}

\begin{proof}
  Let $\varphi,\psi\in\End{\Om AV}$ and $f\in S^A$. Since
  $\widehat{\hat f\circ\varphi|_A}=\hat f\circ\varphi$, we obtain the
  following chain of equalities:
  $$(\varphi\circ\psi)_S(f)
  =\hat f\circ\varphi\circ\psi|_A
  =\widehat{\hat f\circ\varphi|_A}\circ\psi|_A
  =\psi_S(\hat f\circ\varphi|_A)
  =\psi_S\circ\varphi_S(f),
  $$
  which proves that our mapping is an anti-homomorphism. To prove that
  it is continuous, consider a net limit $\varphi=\lim\varphi_i$
  in~\End{\Om AV}. Then, for every $f\in\ci T(S^A)$ and every $a\in A$,
  we may perform the following computation:
  \begin{align*}
    \varphi_S(f)(a)
    &=
    \hat f(\varphi(a))
    =
    \hat f\bigl((\lim\varphi_i)(a)\bigr)
    =
    \hat f(\lim\varphi_i(a))
    \\
    &=
    \lim\hat f(\varphi_i(a))
    =
    \lim(\varphi_i)_S(f)(a),
  \end{align*}
  which yields the desired equality
  $\varphi_S=\lim(\varphi_i)_S$.
\end{proof}

The following result will be useful to draw structural and
computational information about presentations of the
form~\eqref{eq:V-presentation}. To state it, we require some
further terminology. For a semigroup $S$, we say that a mapping $f\in
S^A$ is a \emph{generating mapping} if $f(A)$ generates~$S$. Given a
pseudovariety of semigroups~\pv V, a subpseudovariety \pv W, and an
endomorphism $\varphi$ of $\Om AV$, let $\varphi_{\pv W}$ be the
unique continuous endomorphism of $\Om AW$ such that $\varphi_{\pv
  W}\circ p=p\circ \varphi$, where $p:\Om AV\to\Om AW$ is the
canonical projection. In particular, if $\varphi\in\End{\Om AW}$, then
$\varphi_\pv W=\varphi$.

\begin{prop}
  \label{p:values-omega-powers-subst}
  Let\/ \pv V and \pv W be pseudovarieties of semigroups such that
  $\pv W\subseteq\pv V$. Let $A$~be a finite alphabet and let
  $\varphi$ be a continuous endomorphism of\/~\Om AV. The following
  are equivalent for an arbitrary semigroup $S$ from~\pv W:
  \begin{enumerate}[(1)]
  \item\label{item:values-omega-powers-subst-1} $S$~is a continuous
    homomorphic image of the semigroup presented by
    \begin{equation} 
     \label{eq:values-omega-powers-subst-presentation}
      \langle A\mid \varphi_\pv W^\omega(a)=a\ (a\in A)\rangle_\pv W;
    \end{equation}
  \item\label{item:values-omega-powers-subst-2} there is some
    generating mapping $f:A\to S$ and some integer~$n$ such that $1\le
    n\le|S^A|$ and $\varphi_S^n(f)=f$;
  \item\label{item:values-omega-powers-subst-3} there is some
    generating mapping $f:A\to S$ and some integer~$n$ such that
    $\varphi_S^n(f)=f$.
  \end{enumerate}
\end{prop}

\begin{proof}
  Let $T$ be the profinite semigroup defined by the presentation
  \eqref{eq:values-omega-powers-subst-presentation} and consider the
  natural homomorphisms $p:\Om AV\to\Om AW$ and $\pi:\Om AW\to T$.

  We begin by proving
  (\ref{item:values-omega-powers-subst-1})\,$\Rightarrow$\,(\ref{item:values-omega-powers-subst-2}).
  Suppose that $\theta:T\to S$ is an onto continuous homomorphism.
  Consider the mapping $f=\theta\circ\pi\circ p|_A\in S^A$, whose
  unique continuous homomorphic extension $\hat f:\Om AV\to S$ is the
  mapping $\theta\circ\pi\circ p$. Since $\varphi_\pv W\circ
  p=p\circ\varphi$, we deduce that
  $\varphi_S^k(f)=\theta\circ\pi\circ\varphi_\pv W^k\circ p|_A$ for every
  $k\ge0$, where we write $\varphi^0$ and $\varphi^0_\pv W$ for the
  identity mappings on \Om XV and~\Om XW, respectively. Hence, for
  every $a\in A$, the following equalities hold:
  $\varphi_S^\omega(f)(a) %
  =\theta\circ\pi\circ\varphi_\pv W^\omega(a) %
  =\theta\circ\pi\circ\varphi_\pv W^0(a) %
  =\varphi_S^0(f)(a)=f(a)$. We have thus proved that
  $\varphi_S^\omega(f)=f$. As $\varphi_S$ is a transformation of the
  set $S^A$, the successive iterates
  $f,\varphi_S(f),\varphi_S^2(f),\ldots,\varphi_S^{|S^A|}(f)$ cannot
  all be distinct and $\varphi_S^\omega(f)$ must be found in the
  sequence on the first repeated point or between it and its first
  repetition. Hence, the equality $\varphi_S^\omega(f)=f$ implies that
  $\varphi_S^n(f)=f$ for some integer $n$ such that $1\le n\le |S^A|$.

  The implication
  (\ref{item:values-omega-powers-subst-2})\,$\Rightarrow$\,(\ref{item:values-omega-powers-subst-3})
  being trivial, it remains to prove the implication
  (\ref{item:values-omega-powers-subst-3})\,$\Rightarrow$\,(\ref{item:values-omega-powers-subst-1}).
  It suffices to show that $\hat f$ factors through $\pi\circ p$.
  Since $S\in\pv W$, $\hat f$ factors through $p$, and we have the
  following commutative diagram, where the existence of the dashed
  arrow $\theta$ is yet to be established:
  \begin{equation*}
    \label{eq:values-omega-powers-subst-2}
    \begin{split}
      \xymatrix{
        \Om AV
        \ar[r]^{\hat f}
        \ar[d]_p
        &
        S
        \\
        \Om AW
        \ar[r]^\pi
        \ar[ru]^\eta
        &
        T
        \ar@{-->}[u]_\theta
      }
    \end{split}
  \end{equation*}
  Thus, it is enough to verify that, for every $a\in A$,
  $\eta(\varphi_\pv W^\omega(a))=\eta(a)$. Taking into account the
  definition of $p$, the desired
  equality is equivalent to $\eta(\varphi_\pv
  W^\omega(p(a)))=\eta(p(a))$. In view of $\varphi_\pv W\circ
  p=p\circ\varphi$ and $\eta\circ p=\hat f$, this translates into the
  equality $\hat f(\varphi^\omega(a))=\hat f(a)$. Indeed, by
  hypothesis, we have $\varphi_S^n(f)=f$ for some~$n$, hence $f$~is
  fixed by all powers of~$\varphi_S^n$ and, therefore, also by
  $\varphi_S^\omega=(\varphi_S^n)^\omega$.
\end{proof}

We say that a profinite semigroup $S$ is \emph{decidable} if there is
an algorithm to determine, for a given finite semigroup $T$, whether
there is a continuous homomorphism from $S$ onto~$T$. For instance, if
\pv V is a pseudovariety of semigroups and $A$~is a finite set, then
\Om AV, the pro-\pv V semigroup freely generated by~$A$, is decidable
if and only if it is decidable whether a finite $A$-generated
semigroup belongs to~\pv V. Thus, the pseudovariety \pv V has a
decidable membership problem if and only if all finitely generated
free pro-\pv V semigroups are decidable.

The following immediate application of
Proposition~\ref{p:values-omega-powers-subst} could be stated, and
essentially proved in the same way, for much more general
presentations. To avoid introducing further notation, we stick here to
the type of presentations in which we are mostly interested.

\begin{cor}
  \label{c:special-presentation->decidable}
  Let $\varphi$ be an endomorphism of the free group $FG(A)$ on a
  finite set~$A$ and let $\hat\varphi$ be its unique extension to a
  continuous endomorphism of\/~\Om AG. Then the profinite group
  presented by $\langle A\mid\hat\varphi^\omega(a)=a\ (a\in
  A)\rangle_\pv G$ is decidable.\qed
\end{cor}

\section{Preliminaries on symbolic dynamics}
\label{sec:preliminaries}

Let $A$ be a finite alphabet. We denote by $A^+$ the free semigroup
on~$A$. A \emph{code} is a nonempty subset of~$A^+$ that generates a
free subsemigroup.

The subsemigroup of~\Om AS generated by~$A$ is a free semigroup, and
so we identify it with $A^+$. The elements of $\Om AS\setminus A^+$
are said to be \emph{infinite}, while those of~$A^+$,
which are isolated elements of $\Om AS$, are said to be \emph{finite}.

We may represent an element $x$ of $A^\mathbb{Z}$ as the biinfinite word
$$\cdots x(-2)x(-1)\cdot x(0)x(1)x(2)\cdots.$$
For $x\in A^\mathbb{Z}$ and integers $k,\ell$ with $k\le\ell$, we
denote by $x_{[k,\ell]}$ the word $x(k)x(k+1)\cdots x(\ell)$; a word
of this form is called a \emph{finite block} of~$x$.

A \emph{symbolic dynamical system} \ci X of $\z A$, also called
\emph{subshift} or \emph{shift space} of $\z A$, is a nonempty closed
subset of $\z A$ invariant under the shift operation and its
inverse~\cite{Lind&Marcus:1996}. We denote by $L(\ci X)$ the set of
all finite blocks of elements of $\ci X$.

A subshift \ci X is \emph{minimal} if it does not contain proper
subshifts. There is another useful characterization of minimal
subshifts, with a combinatorial flavor. An element $x\in\z A$ is
\emph{uniformly recurrent} if for every finite block $w$ of $x$, there
is a positive integer $N$ such that $w$ is a factor of every finite
block of $x$ with length~$N$. It turns out that a subshift is minimal
if and only if it is generated by a uniformly recurrent biinfinite
sequence~\cite[Proposition 5.1.13]{Fogg:2002}.

A trivial example of minimal subshift is that of a minimal finite
subshift, generated by a periodic biinfinite word. Such a subshift is
said to be \emph{periodic}.

Given a subshift \ci X and
$u\in L(\ci X)$, say that a nonempty word $v$ is a \emph{return word
  of~$u$ in \ci X} if $vu\in L(\ci X)$, $u$ is a prefix of $vu$ and
$u$ occurs in $vu$ only as a prefix and a suffix. The set of all
return words of~$u$ is denoted $R(u)$.
See~\cite{Balkova&Pelantova&Steiner:2008} for a recent account on
return words. A subshift generated by
an element of $\z A$ is minimal if and only if each of its finite
blocks has a finite set of return words.

The following discussion summarizes results that can be found in
\cite[Section~2]{Almeida:2005c} and
\cite[Section~6]{Almeida&ACosta:2007a}. If the subshift \ci X is
minimal, then the topological closure of $L(\ci X)$ in $\Om AS$ is the
disjoint union of $L(\ci X)$ and a $\ci J$-class $J(\ci X)$ of maximal
regular elements of \Om AS. The correspondence $\ci X\mapsto J(\ci X)$
is a bijection between the set of minimal subshifts of $\z A$ and the
set of maximal regular $\ci J$-classes of $\Om AS$. Moreover, an
infinite element $w$ of~\Om AS belongs to $J(\ci X)$ if and only if
all its finite factors lie in~$L(\ci X)$.

It is natural to ask what is the structure of the (isomorphic) maximal
subgroups of $J(\ci X)$, denoted $G(\ci X)$. Since the expression ``maximal
subgroup of $J(\ci X)$'' refers to a concrete subgroup of the free
profinite semigroup and we wish to investigate its structure as an
abstract profinite group, we prefer to call $G(\ci X)$ the
\emph{Sch\"utzenberger group of~$\ci X$}. This is in accordance with
the literature in semigroup theory in which, more generally, one
associates an abstract group with every \ci D-class of a semigroup,
which is known as its \emph{Sch\"utzenberger group}.

For instance, it is proved in~\cite{Almeida:2005c} that, if \ci X is
an Arnoux-Rauzy subshift of degree~$k$, of which the case $k=2$ is
that of the extensively studied Sturmian subshifts
\cite{Lothaire:2001, Fogg:2002}, then $G(\ci X)$ is a free profinite
group of rank $k$. An example of a minimal subshift \ci X such that
$G(\ci X)$ is not freely generated, with rank two, is also given in
the same paper~\cite[Example 7.2]{Almeida:2005c}.

A \emph{right} (respectively \emph{left}) \emph{infinite word}
is an element of $A^\mathbb{N}$ (resp.~of $A^{\mathbb{Z}^-}$).
Given $w\in\Om AS$, we denote by $\overrightarrow w$ (respectively
$\overleftarrow w$) the right (resp.~left) infinite word whose finite
prefixes (resp.~suffixes) are those of~$w$.

\begin{lema}[{\cite[Lemma~6.6]{Almeida&ACosta:2007a}}]
  \label{l:H-class-vs-x}
  For a minimal subshift $\ci X$, two elements $u,v\in J(\ci X)$ are
  \ci R-equivalent if and only if $\overrightarrow u=\overrightarrow
  v$ and \ci L-equivalent if and only if $\overleftarrow
  u=\overleftarrow v$.
\end{lema}

Taking into account \cite[Lemma~8.2]{Almeida&Volkov:2006}, we deduce
that $w\in J(\ci X)$ lies in a subgroup if and only if the doubly
infinite word $\overleftarrow w\cdot\overrightarrow w$ belongs to~\ci
X. Indeed, $w\in J(\ci X)$ lies in a subgroup if and only if $w^2$
stays in the same \ci J-class, that is it has the same finite factors
as~$w$. Now, by \cite[Lemma~8.2]{Almeida&Volkov:2006}, the finite
factors of~$w^2$ are those of~$w$ together with the products of the
form $uv$, where $u$ is a finite suffix of~$w$ and $v$~is a finite
prefix of~$w$. Thus, altogether, the finite factors of~$w^2$ are the
finite factors of~$\overleftarrow w\cdot\overrightarrow w$.

The maximal subgroups $H$ of $J(\ci X)$ are thus in bijection with
the elements of~\ci X via the mapping that sends $H$ to
$\overleftarrow w\cdot\overrightarrow w$, where $w$ is any element
of~$H$. For $x\in\ci X$, we denote by $H_x$ the maximal subgroup
corresponding to~$x$.

By a \emph{substitution over a finite alphabet $A$} we mean an
endomorphism of the free semigroup $A^+$. The substitution $\varphi$
over the alphabet $A$ is \emph{primitive} if there is a positive
integer $n$ such that, for all $a,b\in A$, $a$ occurs
in~$\varphi^n(b)$ and $\lim|\varphi^n(b)|=\infty$, where $|u|$ denotes
the length of the word $u$. It is well known that to each primitive
substitution $\varphi$ over a finite alphabet $A$, we can associate a
minimal subshift $\ci X_\varphi$. In terms of biinfinite words, there
are some such words that are periodic for the action of~$\varphi$
given by
$$x\mapsto \cdots\varphi(x(-2))\varphi(x(-1))\cdot
\varphi(x(0))\varphi(x(1))\varphi(x(2))\cdots $$
(cf.~\cite[Exercise~1.2.1]{Fogg:2002}) and the subshift $\ci
X_\varphi$~is generated by it. A finite word belongs to the language
$L(\ci X_\varphi)$ if and only if it is a factor of $\varphi^k(a)$ for
all $a\in A$ and all sufficiently large $k\ge1$. Note that
$\varphi(L(\ci X_\varphi))\subseteq L(\ci X_\varphi)$ (see, for
instance \cite[Lemma~4.1(a)]{Almeida:2005c}). We say that $\varphi$ is
\emph{periodic} in case so is $\ci X_\varphi$.

We shall denote $J(\ci X_\varphi)$ and $G(\ci X_\varphi)$ respectively
by $J(\varphi)$ and $G(\varphi)$: this notation is more synthetic and
emphasizes the exclusive dependence of these structures on $\varphi$,
which in turn is a mathematical object completely determined by a
finite amount of data, namely the images (in $A^+$) of letters
by~$\varphi$. Naturally, we also call $G(\varphi)$ the
\emph{Sch\"utzenberger group} of the primitive substitution~$\varphi$.

The unique continuous endomorphism of $\Om AS$ extending $\varphi$
will also be denoted by $\varphi$. A \emph{connection for $\varphi$}
is a word $ba$, with $b,a\in A$, such that $ba\in L(\ci X_\varphi)$,
the first letter of $\varphi^\omega(a)$ is $a$, and the last letter of
$\varphi^\omega(b)$ is $b$. Every primitive substitution has a
connection \cite[Corollary~4.12]{Almeida:2005c}. In terms of the
subshift $\ci X_\varphi$, a connection is simply a word of the form
$x(-1)x(0)$ for some periodic point $x$ of the action of~$\varphi$ on
biinfinite words. For a connection $ba$, the intersection $H_{ba}$ of
the $\ci R$-class containing $\varphi^\omega(a)$ with the $\ci
L$-class containing $\varphi^\omega(b)$ is a maximal subgroup of
$J(\varphi)$. There is a finite power $\tilde\varphi$ of $\varphi$
such that the first letter of $\tilde\varphi(a)$ is $a$ and the last
letter of $\tilde\varphi(b)$ is $b$. We call $\tilde\varphi$ a
\emph{connective power} of~$\varphi$ (with respect to the
connection~$ba$).

We let $X_\varphi(a,b)=b^{-1}(R(ba)b)$.
To avoid overloaded notation, $X_\varphi(a,b)$ will be usually denoted by $X$.
The set $R(ba)$ is easily recognized to be a code and so is
$X=b^{-1}(R(ba)b)$. Let $i$ be the unique homomorphism from the
semigroup freely generated by $X$ into the semigroup freely generated
by $A$ such that $i(x)=x$ for all $x\in X$. Then $i$ is injective,
because $X$ is a code. If $x\in X$ then $\tilde\varphi(x)$ belongs to
the subsemigroup of $A^+$ generated by $X$. Therefore, we can consider
the word $w_x=i^{-1}(\tilde\varphi(x))$, the unique decomposition of
$\tilde\varphi(x)$ in the elements of $X$. The homomorphism $i$ has a
unique extension to a continuous homomorphism $\Om {X}S\to \Om AS$,
which we also denote by~$i$, and which we call the \emph{encoding
  associated with} the connection $ba$.

\begin{teor}[{\cite[Corollary~2.2]{Margolis&Sapir&Weil:1995}}]
  \label{t:MSW}
  The mapping $i$~is injective.
\end{teor}

Let $q$ be the canonical projection $\Om {X}S\to \Om {X}G$, namely the
unique continuous homomorphism from $\Om {X}S$ into $\Om {X}G$ that is
the identity on the generators. Then there are unique continuous
endomorphisms $\tilde\varphi_X$ and $\tilde\varphi_{X,\pv G}$ such
that Diagram~\eqref{eq:hat-varphi} commutes. More explicitly, for each
$x\in X$ we have $\tilde\varphi_X(x)=w_x$ and $\tilde\varphi_{X,\pv
  G}(x)=w_x$, where we regard $w_x$ as a semigroup word and a group word,
respectively.
\begin{equation}\label{eq:hat-varphi}
  \begin{split}
    \xymatrix{
      \Om AS
      \ar[d]_{\tilde\varphi}
      &
      \Om {X}S
      \ar[l]_i
      \ar[r]^q
      \ar[d]^{\tilde\varphi_X}
      &
      \Om {X}G
      \ar[d]^{\tilde\varphi_{X,\pv G}}
      \\
      \Om AS
      &
      \Om {X}S
      \ar[l]_i
      \ar[r]^q
      &
      \Om {X}G
    }    
  \end{split}
\end{equation}

\section{Maximal
subgroups fixed by powers of primitive substitutions}
\label{sec:max-subgps-as-retracts}

Let $A$ be a finite alphabet and let $\ci X\subseteq A^\mathbb{Z}$ be
a subshift. Given a word $u\in L(\ci X)$, let $n$ be a nonnegative
integer less than or equal to the length of~$u$. Let $u_1$ and $u_2$
be words such that $u=u_1u_2$ and $|u_1|=n$. An \emph{$n$-delayed
  return word of~$u$ in~\ci X} is a word $v$ such that $u_1vu_2\in
L(\ci X)$ and $u_1v\in R(u)u_1$ (see
\cite[Definition~11]{Durand&Host&Skau:1999}). The set of $n$-delayed
return words of~$u$ in~\ci X shall be denoted by $R(n,u)$ or
$R(u_1,u_2)$. Note that
\begin{equation*}
  \label{eq:return-words-versus-generalized-ones}
  R(u_1,u_2)=u_1^{-1}(R(u)u_1),
\end{equation*}
thus $R(u_1,u_2)$ and $R(u)$ have the same cardinality.

\begin{lema}\label{l:return-words}
  Let $\ci X$ be a minimal subshift of \z A. Let $v\in\Om AS$ be an
  element of a maximal subgroup of $J(\ci X)$. If $u_1$ and $u_2$ are
  words such that $u_1$ is a suffix and $u_2$ is a prefix of $v$, then
  $v$ belongs to
  $\overline{\langle R(u_1,u_2)\rangle}$.
\end{lema}

\begin{proof}
  Since $v$ lies in a subgroup, $v^3$ also belongs to $J(\ci X)$,
  whence so does $u_1vu_2$.
  The set $\overline{L(\ci X)}$ is closed under taking factors
  by~\cite[Proposition~2.4]{Almeida&ACosta:2007a}, and so there is a sequence
  $(w_n)$ of words in~$L(\ci X)$ that converges to~$u_1vu_2$. We may
  as well assume that $u_1u_2$ is a prefix and a suffix of each $w_n$.
  Hence, $w_n(u_1u_2)^{-1}$ is a product of words in~$R(u_1u_2)$ and,
  therefore, $u_1^{-1}w_nu_2^{-1}$ belongs to the subsemigroup
  generated by $R(u_1,u_2)$, from which the lemma follows by
  \cite[Exercise~10.2.10]{Almeida:1994a}.
\end{proof}

We recall that the evaluation mapping
\begin{eqnarray}\label{eq:evaluation-mapping-is-continuous}
  \Om MS\times (\Om AS)^{M}&\to&\Om AS\\
  (w,v_1,\ldots,v_M)&\mapsto&w(v_1,\ldots,v_M)\notag
\end{eqnarray}
is continuous for every positive integer
$M$~\cite[Subsection~2.3]{Almeida:1999c}.

\begin{prop}
  \label{p:generation-by-limits-of-return-words}
  Let $\ci X$ be a minimal non-periodic subshift of \z A and let
  $x\in\ci X$. Suppose there are strictly increasing sequences of
  positive integers $(p_n)_n$ and $(q_n)_n$ such that
  $R(p_n,x_{[-p_n,q_n]})$ has exactly $M$ elements
  $r_{n,1},\ldots,r_{n,M}$, for every $n$. Let $(r_{1},\ldots,r_{M})$
  be an arbitrary accumulation point of the sequence
  $(r_{n,1},\ldots,r_{n,M})_{n}$ in $(\Om AS)^M$. Then
  $\overline{\langle r_{1},\ldots,r_{M}\rangle}$ is the maximal
  subgroup $H_x$ of $J(\ci X)$.
\end{prop}

\begin{proof}
  Clearly the proof needs only to deal with the case where the
  sequence $(r_{n,1},\ldots,r_{n,M})_{n}$ converges to
  $(r_{1},\ldots,r_{M})$. Since $r_{n,i}\in L(\ci X)$ for all $n$, we
  know that $r_i\in\overline{L(\ci X)}$. Let
  $p_{(n,i)}=\min\{p_n,|r_{n,i}|\}$ and
  $q_{(n,i)}=\min\{q_n+1,|r_{n,i}|\}$. By
  \cite[Lemma~3.2]{Durand:1998},
  \begin{equation*}
    \lim\min\{|v|: v\in R(p_n,x_{[-p_n,q_n]})\}=\infty.
  \end{equation*}
  Hence, $r_i\in J(\ci X)$ and $\lim p_{(n,i)}=
  \lim q_{(n,i)}= \infty$. Since for all~$n$ the word
  $x_{[0,q_{(n,i)}]}$ is a prefix of~$r_{n,i}$ and
  $x_{[-p_{(n,i)},-1]}$ is a suffix of~$r_{(n,i)}$, we obtain
  $r_i\in H_x$ by definition of~$H_x$.

  Let $g$ be an element of~$H_x$. By Lemma~\ref{l:return-words}, there
  is an element $w_n\in\Om MS$ such that
  $g=w_n(r_{n,1},\ldots,r_{n,M})$. Let $w$ be an accumulation point of
  $(w_n)_n$ in $\Om MS$. Since the evaluation
  mapping~\eqref{eq:evaluation-mapping-is-continuous} is continuous,
  it follows that $g=w(r_1,\ldots,r_M)$. This proves that
  $H_x=\overline{\langle r_{1},\ldots,r_{M}\rangle}$.
\end{proof}

The following result shows that a primitive substitution $\varphi$
induces natural actions on certain maximal subgroups of~$J(\varphi)$.

\begin{lema}
  \label{l:back-to-H}
  Let $\varphi$ be a primitive substitution and let $ba$ be a
  connection for~$\varphi$.  If $\tilde\varphi$ is a connective power
  of~$\varphi$, then $\tilde\varphi(H_{ba})\subseteq H_{ba}$.
\end{lema}

\begin{proof}
  Let $u$ be a word from~$R(b,a)$. Then $\varphi^n(u)$ belongs to
  $L(\ci X_\varphi)$ for every~$n$. Hence, $\varphi^\omega(u)$ belongs
  to~$J(\varphi)$. Since $u$ starts with $a$ and ends with~$b$, it
  follows that $\varphi^\omega(u)\in H_{ba}$.
  Let $K=\tilde\varphi(H_{ba})$. Since, by
  \cite[Proposition~4.2]{Almeida:2005c}, $\varphi$ maps $J(\varphi)$
  to itself, $K$ is a subgroup of~\Om AS contained in~$J(\varphi)$.
  Thus, since $H_{ba}$~is a maximal subgroup of~\Om AS, to show that
  $\tilde\varphi(H_{ba})\subseteq H_{ba}$, it suffices to show that
  $K\cap H_{ba}$ is  nonempty. Indeed,
  $\tilde\varphi\varphi^\omega(u)=\varphi^\omega\tilde\varphi(u)$
  belongs to~$K$, by definition of~$K$, and to~$H_{ba}$, since
  $\tilde\varphi$ is a connective power of~$\varphi$.
\end{proof}

Let $\varphi$ be a primitive substitution over $A$.
A \emph{biinfinite fixed point of $\varphi$} is an element $x$
of $A^\mathbb{Z}$ such that
$x_{[0,n]}$ is a prefix of $\varphi(x_{[0,n]})$
and
$x_{[-n,-1]}$ is a suffix of $\varphi(x_{[-n,-1]})$, for every positive
integer $n$.

Given a biinfinite word $x\in A^\mathbb{Z}$ and a positive integer
$\ell$, let $\sim_\ell$ be the equivalence relation on~$\mathbb{Z}$
defined by $i\sim_\ell j$ if
$x_{[i-\ell,i+\ell]}=x_{[j-\ell,j+\ell]}$. Note that, for $k>\ell$,
$\sim_k$ refines $\sim_\ell$.

Suppose additionally that $x$ is a biinfinite fixed point of
$\varphi$. Following the notation of~\cite{Fogg:2002},\footnote{The
  minus sign in~$-|\varphi(x_{[-n,-1]})|$ in the formula for
  $E_1(\varphi)$ is missing in~\cite[Section~7.2.1]{Fogg:2002}. The
  correct formulation can be found
  in~\cite[Section~2.4]{Durand&Host&Skau:1999}.} let
$$E_1(\varphi)=\{0\}\cup
\bigcup_{n\ge
  1}\bigl\{-|\varphi(x_{[-n,-1]})|,|\varphi(x_{[0,n-1]})|\bigr\}.
$$
The substitution $\varphi$ is said to be
\emph{bilaterally recognizable} if there exists
$\ell>0$ such that $E_1$ is a union of $\sim_\ell$-classes.
Denote by $\ell(\varphi)$ the least possible value of $\ell$.

The following result of Moss\'e~\cite{Mosse:1992,Mosse:1996},
stated in \cite[Theorem 7.2.2]{Fogg:2002},
will be crucial in the sequel.

\begin{teor}\label{t:mosse}
  Every non-periodic primitive substitution with a biinfinite fixed
  point is bilaterally recognizable.
\end{teor}

In the case of non-periodic primitive substitutions, the following
consequence of Theorem~\ref{t:mosse} provides the key tool to prove
the reverse inclusion of that given by Lemma~\ref{l:back-to-H}.

\begin{prop}
  \label{p:in-Im-phi}
  Let $\varphi$ be a non-periodic primitive substitution and let $ba$ be a
  connection for~$\varphi$. If $\tilde\varphi$ is a connective power
  of~$\varphi$, then $H_{ba}\subseteq \Im{\tilde\varphi}$.
\end{prop}

\begin{proof}
  Let $x$ be the unique element of the subshift~$\ci X_\varphi$ such
  that
  $H_x=H_{ba}$. By Lemma~\ref{l:back-to-H}, $x$ is also the biinfinite word
  $$\cdots\tilde\varphi(x(-2))\tilde\varphi(x(-1))\cdot
  \tilde\varphi(x(0))\tilde\varphi(x(1))\tilde\varphi(x(2))\cdots.$$
  Therefore, $x$ is a biinfinite fixed point of $\tilde\varphi$.
  By~Theorem~\ref{t:mosse}, $\tilde\varphi$ is bilaterally
  recognizable. 
  
  By~\cite[Proposition 25 and Theorem 24]{Durand&Host&Skau:1999},
  the sequence
  $|R(n,x_{[-n,n]})|$ is bounded. Hence, there is a strictly increasing
  sequence $(p_n)$ for which
  $|R(p_n,x_{[-p_n,p_n]})|$ is
  a constant $M$, and such that $p_n>\ell(\tilde\varphi)$ for all $n$.

  Let $R(p_n,x_{[-p_n,p_n]})=\{r_{n,1},\ldots,r_{n,M}\}$.
    Let $k\in\{1,\ldots, M\}$.
    Because $x$
    is uniformly recurrent,
    there are $i>0$ and $j>i$ such that
    \begin{equation}
      \label{eq:recurrence}
      x_{[-p_n,p_n]}=x_{[i-p_n,i+p_n]}=x_{[j-p_n,j+p_n]}
    \end{equation}
    and $r_{n,k}=x_{[i,j-1]}$. Since $0\in E_1(\tilde\varphi)$ and
    $p_n>\ell(\tilde\varphi)$, it follows from (\ref{eq:recurrence})
    that $i,j\in E_1(\tilde\varphi)$. As $r_{n,k}=x_{[i,j-1]}$, we
    conclude that $r_{n,k}$ belongs to $\Im{\tilde\varphi}$.

    Let $(r_{1},\ldots,r_{M})$ be an accumulation point of the
    sequence $(r_{n,1},\ldots,r_{n,M})_{n}$. Since
    $\Im{\tilde\varphi}$ is closed in $\Om AS$, we have $r_k\in
    \Im{\tilde\varphi}$ for all $k$. It then follows from
    Proposition~\ref{p:generation-by-limits-of-return-words} that
    $H_{ba}\subseteq \Im{\tilde\varphi}$.
\end{proof}

We can now establish the announced reverse inclusion of that given by
Lemma~\ref{l:back-to-H} in the case of non-periodic primitive
substitutions.

\begin{teor}
  \label{t:in-Im-phi}
  Let $\varphi$ be a non-periodic primitive substitution.
  Consider a connection $ba$ for $\varphi$
  and a connective power~$\tilde\varphi$.
  Then $H_{ba}=\tilde\varphi(H_{ba})=\varphi^\omega(H_{ba})$.
\end{teor}

\begin{proof}
  Let $k$ be a positive integer. Then $\tilde\varphi^{k}$ is also a
  connective power of $\varphi$ relatively to the connection $ba$.
  Therefore, by Proposition~\ref{p:in-Im-phi}, we obtain the
  inclusion~$H_{ba}\subseteq \Im {\tilde\varphi^{k}}$. Hence, given
  $g\in H_{ba}$, for each positive integer~$k$, there is $u_k\in\Om
  AS$ such that $g=(\tilde\varphi)^{k!}(u_k)$. Since the evaluation
  mapping on continuous endomorphisms of finitely generated profinite
  semigroups is continuous (cf.~\cite[Proposition~1]{Hunter:1983}),
  there is an accumulation point $u$ of the sequence $(u_k)$ such that
  $g=\tilde\varphi^\omega(u)=\varphi^\omega(u)$. Hence, we have
  $\varphi^{\omega}(g)=g$. This proves the equality
  $H_{ba}=\varphi^\omega(H_{ba})$.
  
  By Lemma~\ref{l:back-to-H}, the inclusion
  $\tilde\varphi(H_{ba})\subseteq H_{ba}$ holds.
  Hence, $\tilde\varphi^{k!-1}(H_{ba})\subseteq H_{ba}$ holds
  for all $k\ge 1$, which shows that
  $\tilde\varphi^{\omega-1}(H_{ba})\subseteq H_{ba}$ because
  $H_{ba}$ is closed.
  We then have
  $$H_{ba}=\varphi^\omega(H_{ba})=
  \tilde\varphi(\tilde\varphi^{\omega-1}(H_{ba}))
  \subseteq \tilde\varphi(H_{ba}),
  $$
  which, together with
  Lemma~\ref{l:back-to-H},
  establishes the equality
  $H_{ba}=\tilde\varphi(H_{ba})$.
\end{proof}

The first author \cite[Theorem~4.13]{Almeida:2005c} managed to avoid
using Moss\'e's Theorem~\ref{t:mosse} to obtain the equality
$H_{ba}=\varphi^\omega(H_{ba})$ by adding the extra synchronization hypothesis
that $\varphi$ is an ``encoding of bounded delay with respect to the
finite factors of~$J(\varphi)$'' (cf.~\cite{Almeida:2005c}). This
restriction turns out not to be significant in case $\varphi$ induces
an automorphism of the free group $FG(A)$, because then the extra
hypothesis always holds \cite[Corollary~5.6]{Almeida:2005c}.

Theorem~\ref{t:in-Im-phi}
fails if $\varphi$ is periodic. For example, consider the periodic primitive
substitution defined by $\varphi(a)=aba$ and $\varphi(b)=bab$.
Then $ba$ is a connection for $\varphi$, and $X_\varphi(a,b)=\{ab\}$.
Note that $\varphi^n(ab)=(ab)^{3^n}$, for every positive integer $n$.
By the definition of
$H_{ba}$, we know that $K=\overline{\langle\varphi^\omega(ab)\rangle}$ is a
closed subgroup of $H_{ba}$. 
Note that $(ab)^{\omega+1}$ is $\ci H$-equivalent to 
$\varphi^\omega(ab)$, that is $(ab)^{\omega+1}\in H_{ba}$.
Note also that $(ab)^{\omega+1}\notin K$.
If we had $H_{ba}\subseteq \Im{\varphi^\omega}$, then
we would have
$(ab)^{\omega+1}=\varphi^\omega(ab)^{\omega+1}
\in K$, a contradiction. This shows the necessity of the
non-periodicity hypothesis in Theorem~\ref{t:in-Im-phi}.

\section{Presentations of Sch\"utzenberger groups of primitive
  substitutions}
\label{sec:core}

By Corollary~\ref{c:there-is-some-presentation-of-that-kind}, every
finitely generated maximal subgroup of~\Om XS admits a finite
presentation of the form \eqref{eq:G-presentation}. However, to be
able to apply the decidability results of
Section~\ref{sec:decidability}, one needs computability properties of
the continuous endomorphism $\Phi$ of~\Om XG. In this section, we show
that this is always possible for the Sch\"utzenberger group of an
arbitrary primitive substitution over a finite alphabet.

We separate into two subsections the general case, which involves
return words, and a special case, in which the idempotent iterate
$\varphi^\omega$ of the substitution $\varphi$ maps all letters to the
same \ci H-class. In the special case, the presentation can be
expressed more directly in terms of the given substitution. In the
third subsection, we show that one can actually obtain the general
case from the special one.

\subsection{The general case}
\label{sec:general-substitutions}

We first apply the simple remarks of
Subsection~\ref{sec:presentations-general} to obtain a semigroup
presentation for a profinite subgroup associated with a primitive
substitution $\varphi$ and a connection of~$\varphi$.

\begin{prop}
  \label{p:generic}
  Let $\varphi$ be a primitive substitution over the alphabet~$A$,
  $ba$ be a connection for~$\varphi$, and $\tilde\varphi$ be a
  connective power of~$\varphi$. Put $X=X_\varphi(a,b)$ and
  $H=\Im{(\varphi^\omega\circ i)}$, where $i$ is the encoding
  associated with $ba$. Then $\KerS{(\varphi^\omega\circ i)}
  \subseteq\KerS{\tilde\varphi_X^\omega}$ and so $H$ admits the
  presentation
  \begin{equation}
    \label{eq:semigroup-presentation}
    \langle X\mid \tilde\varphi_X^\omega(x)=x\ (x\in X)\rangle_\pv S.
  \end{equation}
\end{prop}

\begin{proof}
  Note that $\varphi^\omega(i(X))$ is contained in~$H_{ba}$
  \cite[Proposition~4.8(1)]{Almeida:2005c}, whence $H$ is a subgroup
  of~\Om AS. Moreover, $\tilde\varphi$ acts as an automorphism on~$H$
  and as an endomorphism of~$\Im i$. We obtain the following
  commutative diagram, where the commutativity of the outer rectangle
  follows from that of the largest trapezoid.
  $$
  \xymatrix@R=1.7em@C=1.9em{
    \Om XS
    \ar[rrr]^{\tilde\varphi_X^{\omega+1}}
    \ar[ddd]_{\tilde\varphi^\omega\circ i}
    &&&
    \Om XS
    \ar[ddd]^{\tilde\varphi^\omega\circ i}
    \\
    &
    \Om XS
    \ar[r]^{\tilde\varphi_X}
    \ar[d]^i
    \ar[ldd]_{\tilde\varphi^\omega\circ i}
    &
    \Om XS
    \ar[d]_i
    \ar[rdd]^{\tilde\varphi^\omega\circ i}
    &
    \\
    &
    \Im i
    \ar[r]^{\tilde\varphi}
    \ar[ld]^{\tilde\varphi^\omega}
    &
    \Im i
    \ar[rd]_{\tilde\varphi^\omega}
    &
    \\
    H
    \ar[rrr]_{\tilde\varphi=\tilde\varphi^{\omega+1}}
    &&&
    H
  }
  $$
  Let $(u,v)\in\KerS(\tilde\varphi^\omega\circ i)$. We claim that
  $(u,v)\in\KerS {\tilde\varphi_X^\omega}
  =\KerS{\tilde\varphi_X^{\omega+1}}$. Indeed, since $i$~is injective
  by Theorem~\ref{t:MSW}, it suffices to show that $u$ and $v$ have
  the same image under $i\circ\tilde\varphi_X^\omega$. Now, by the
  commutativity of the diagram, the following holds for an arbitrary
  $w\in\Om XS$: $i\circ\tilde\varphi_X^\omega(w)
  =\tilde\varphi^\omega\circ i(w)$. Combining with the hypothesis that
  $\tilde\varphi^\omega\circ i(u)=\tilde\varphi^\omega\circ i(v)$, we
  deduce that
  $i\circ\tilde\varphi_X^\omega(u)=i\circ\tilde\varphi_X^\omega(v)$.
  We have thus shown that $\KerS(\tilde\varphi^\omega\circ i)
  \subseteq\KerS{\tilde\varphi_X^\omega}$. To conclude the proof, it
  suffices to invoke Lemma~\ref{l:equivalent-conditions-s}.
\end{proof}

We are now ready for the main theorem of this paper.

\begin{teor}
  \label{t:conjecture}
  Let $\varphi$ be a non-periodic primitive substitution over the
  alphabet $A$. Let $ba$ be a connection of~$\varphi$ and let
  $X=X_\varphi(a,b)$. Then $G(\varphi)$ admits the presentation
  \begin{equation}
    \label{eq:conjecture-presentation}
    \langle X\mid \tilde\varphi_{X,\pv G}^\omega(x)=x\ (x\in
    X)\rangle_\pv G,
  \end{equation}
  where $\tilde\varphi$ is a connective power of~$\varphi$.
\end{teor}

\begin{proof}
  Let $H=\Im{(\varphi^\omega\circ i)}$. As in the proof of
  Proposition~\ref{p:generic}, we know that $H$ is contained
  in~$H_{ba}$. On the other hand, by Lemma~\ref{l:return-words},
  $H_{ba}$ is contained in~$\Im i$, whence
  $\varphi^\omega(H_{ba})\subseteq\Im{(\varphi^\omega\circ i)}=H$. By
  Theorem~\ref{t:in-Im-phi}, it follows that $H=H_{ba}$. Hence, $H$ is
  the Sch\"utzenberger group $G(\varphi)$. According to
  Proposition~\ref{p:generic}, $H$ admits the profinite semigroup
  presentation~\eqref{eq:semigroup-presentation}.
  Lemma~\ref{l:folklore} yields that $H$ admits the presentation
  $$\langle X\mid q(\tilde\varphi_X^\omega(x))=q(x)\ (x\in
  X)\rangle_\pv G,$$
  where $q:\Om XS\to\Om XG$ is the canonical projection. In view of
  commutativity of Diagram~\eqref{eq:hat-varphi}, and noting also that
  $q(x)=x$ for each $x\in X$, it remains to observe that the above
  presentation is just a reformulation
  of~\eqref{eq:conjecture-presentation}.
\end{proof}

\subsection{The case of proper substitutions}
\label{sec:proper-substitutions}

This subsection is dedicated to a special case in which the
Sch\"utzenberger group of the primitive substitution is realized as a
retract of the free profinite semigroup under the $\omega$-power of
the substitution. This leads to a somewhat simpler presentation of the
form~\eqref{eq:presentation}.

We say that a substitution $\varphi$ over a finite alphabet $A$ is
\emph{proper} if there are letters $a,b\in A$ such that, for every
$d\in A$, the word $\varphi(d)$ starts with $a$ and ends with~$b$.

\begin{lema}
  \label{l:proper-image-maximal-subgroup}
  Let $\varphi$ be a non-periodic proper substitution over a finite
  alphabet $A$. Then $\Im{\varphi^\omega}$ is a maximal subgroup
  of~\Om AS contained in~$J(\varphi)$.
\end{lema}

\begin{proof}
  By \cite[Proposition~5.3]{Almeida&Volkov:2006}, all the elements
  of~\Om AS of the form $\varphi^\omega(a)$ ($a\in A$) lie in the same
  maximal subgroup $H$. By Theorem~\ref{t:in-Im-phi}, the image of
  $\varphi^\omega$ is $H$.
\end{proof}

The special case of the following result where $\varphi$ is an
encoding of bounded delay with respect to the finite factors of
$J(\varphi)$ was announced in a lecture by the first author at the
\emph{Fields Workshop on Profinite Groups and Applications} (Carleton
University, August 2005). Its proof appears here for the first time,
and furthermore does not depend on that hypothesis.

\begin{teor}
  \label{t:presentation-AV-case}
  Let $\varphi$ be a non-periodic proper primitive substitution over a
  finite alphabet~$A$.
  Then $G(\varphi)$ admits the presentation
  \begin{equation}
    \label{eq:presentation-AV-case}
      \langle A\mid \varphi_{\pv G}^\omega(a)=a\ (a\in A)\rangle_\pv G.
  \end{equation}
\end{teor}

\begin{proof}
  By Lemma~\ref{l:proper-image-maximal-subgroup},
  $H=\Im{\varphi^\omega}$ is a maximal subgroup of~\Om AS and it is
  also the Sch\"utzenberger group of~$\varphi$. In particular,
  $\varphi$ acts on~$H$ as an automorphism.

  Consider the following commutative diagram
  $$
  \xymatrix{
    \Om AS
    \ar[r]^{\varphi^{\omega+1}}
     \ar[d]_{\varphi^\omega}
    &
    \Om AS
    \ar[d]^{\varphi^\omega}
    \\
    H
    \ar[r]^\varphi
    &
   \, H.
  }
  $$
  Since $\KerS\varphi^\omega\subseteq\KerS\varphi^{\omega+1}$, it
  follows from Lemma~\ref{l:equivalent-conditions-s} that $H$ admits
  the presentation $\langle A\mid \varphi^\omega(a)=a\ (a\in
  A)\rangle_\pv S$. Applying Lemma~\ref{l:folklore}, we deduce that
  $H$ can also be presented as $\langle A\mid
  p(\varphi^\omega(a))=p(a)\ (a\in A)\rangle_\pv G$. Finally, since
  $p\circ\varphi=\varphi_\pv G\circ p$ and $p(a)=a$ for every $a\in
  A$, the latter presentation coincides with
  \eqref{eq:presentation-AV-case}.
\end{proof}

\subsection{A reduction to the proper case}
\label{sec:reduction-to-proper-case}

In this subsection, we show how to reduce the case of a general
primitive substitution to that of a proper primitive substitution,
thus providing an alternative proof of Theorem~\ref{t:conjecture}
based on Theorem~\ref{t:presentation-AV-case}. The first ingredient is
the following lemma, which can be extracted from
\cite[Lemma~21]{Durand&Host&Skau:1999}, noting that the definition of
proper substitution adopted in that paper translates in the language
of the present paper as a substitution which admits a power which is
proper.

\begin{lema}
  \label{l:phiX-proper-primitive}
  Let $\varphi$ be a primitive substitution over a finite alphabet $A$
  and let $ba$ be a connection for~$\varphi$. Let $X=X_\varphi(a,b)$
  and suppose that $\tilde\varphi$ is a connective power of~$\varphi$
  such that
  $\min\{|\tilde\varphi(a)|,|\tilde\varphi(b)|\}>\max\{|x|:x\in X\}$.
  Then $\tilde\varphi_X$ is a proper primitive substitution over the
  alphabet $X$. The subshift $\ci X_\varphi\subseteq A^\mathbb{Z}$ is
  periodic if and only if so is $\ci X_{\tilde\varphi_X}\subseteq
  X^\mathbb{Z}$.
\end{lema}

Provided $\varphi$ is non-periodic, for a choice of $\tilde\varphi$ as
in Lemma~\ref{l:phiX-proper-primitive}, we may apply 
Lemma~\ref{l:proper-image-maximal-subgroup} to conclude that
$\Im{\tilde\varphi_X^\omega}=\Im{\varphi_X^\omega}$ is a maximal
subgroup of $J(\varphi_X)$, which we denote by~$K$.

On the other hand, Theorem~\ref{t:in-Im-phi} shows that, for a
connection $ba$ of~$\varphi$, the maximal subgroup $H_{ba}$
is such that $H_{ba}=\varphi^\omega(H_{ba})$.
Hence, for the encoding $i$ associated with~$ba$, we have
$$H_{ba}=\varphi^\omega(H_{ba})\subseteq \varphi^\omega(i(\Om XS))
=i(\varphi_X^\omega(\Om XS))=i(K).$$
Since $K$ is a subgroup and $H_{ba}$ is a maximal subgroup, we have
$i(K)=H_{ba}$. As $i$ is injective by Theorem~\ref{t:MSW}, we obtain the
following result.

\begin{teor}
  \label{t:Gphi-vs-GphiX}
  Let $\varphi$ be a non-periodic primitive substitution over a finite
  alphabet and let $\tilde\varphi_X$ be as in
  Lemma~\ref{l:phiX-proper-primitive}. Then the encoding $i$
  associated with the connection~$ba$ defines an isomorphism between a
  maximal subgroup of~$J(\ci X_{\tilde\varphi_X})$ and~$H_{ba}$.\qed
\end{teor}

Combining Theorems~\ref{t:presentation-AV-case}
and~\ref{t:Gphi-vs-GphiX}, it is now immediate to obtain an
alternative way to deduce Theorem~\ref{t:conjecture}.

\section{Applications}
\label{sec:applications}

This section is devoted to applications of the main results of
Section~\ref{sec:core}.

\subsection{Decidability of Sch\"utzenberger groups of primitive substitutions}
\label{sec:dec-Schutz-groups}

The following result is our main motivation for obtaining
presentations of Sch\"utzenberger groups. The periodic case of the
following theorem follows from the fact that the corresponding
Sch\"utzenberger group is a free procyclic group
\cite[Theorem~7.5]{Almeida&Volkov:2006}. For this reason, the two
alternative proofs presented below handle only the non-periodic case.

\begin{teor}
  \label{t:dec-Schutz-groups}
  Let $\varphi$ be a primitive substitution over a finite alphabet.
  Then the profinite group $G(\varphi)$ is decidable.
\end{teor}

\begin{proof}[First proof]
  By~\cite[Lemmas~3.3 and~4.5]{Almeida:2005c}, one may effectively
  compute a connection $ba$ for~$\varphi$ and the set
  $X=X_\varphi(a,b)$. Thus, to conclude the proof, it suffices to
  invoke Corollary~\ref{c:special-presentation->decidable} taking into
  account Theorem~\ref{t:conjecture}.
\end{proof}

\begin{proof}[Second proof]
  Assuming that the substitution $\varphi$~is proper,
  Corollary~\ref{c:special-presentation->decidable} combined with
  Theorem~\ref{t:presentation-AV-case}, yields that the group
  $G(\varphi)$ is decidable. To obtain the general case, we invoke a
  result from symbolic dynamics which states that, for every primitive
  substitution $\varphi$, one can effectively compute a proper
  primitive substitution $\psi$ such that the subshifts $\ci
  X_\varphi$ and $\ci X_\psi$ are conjugate
  \cite{Durand&Host&Skau:1999} (see also
  \cite[Proposition~31]{Durand:2000}). Since the profinite groups
  $G(\varphi)$ and $G(\psi)$ are isomorphic by
  \cite[Theorem~3.11]{ACosta:2006}, and the latter is decidable, so is
  the former.
\end{proof}

Note that, while the first proof depends less on results on symbolic
dynamics, the second proof does not depend on the injectivity of the
encoding $i$ associated with the connection (Theorem~\ref{t:MSW}). By
using also the injectivity of~$i$, one may modify the second proof by
applying instead Theorem~\ref{t:Gphi-vs-GphiX} to obtain the
isomorphism of $G(\varphi)$ with a decidable profinite group.

\subsection{A first non-relatively free example}
\label{sec:ab-aaab}

We give an example to illustrate how to apply
Theorem~\ref{t:presentation-AV-case} to prove that the
Sch\"utzenberger group of a primitive substitution is not relatively
free. Let $A=\{a,b\}$ and define a substitution $\varphi$ by
$\varphi(a)=ab$ and $\varphi(b)=a^3b$, which is non-periodic and
proper primitive. Hence, by Theorem~\ref{t:presentation-AV-case}, the
group $G(\varphi)$ admits the presentation
$$
\langle a,b\mid 
\varphi_{\pv G}^\omega(a)=a,\,
\varphi_{\pv G}^\omega(b)=b
\rangle_\pv G.
$$
It is shown in~\cite[Example~7.2]{Almeida:2005c} that $G(\varphi)$ is
not a free profinite group. We proceed to improve this result by
showing that it is not relatively free, that is, not of the form \Om
XV, although in fact we do not know whether the pseudovariety
generated by all its finite continuous homomorphic images is a proper
subclass of~\pv G.

\begin{teor}
  \label{t:eg:1}
  Let $\varphi$ be the substitution given by $\varphi(a)=ab$ and
  $\varphi(b)=a^3b$. Then $G(\varphi)$ is not a relatively free
  profinite group.
\end{teor}

\begin{proof}
  By Lemma~\ref{l:proper-image-maximal-subgroup}, the closed
  subsemigroup $H=\overline{\langle\varphi^\omega(a),\varphi^\omega(b)\rangle}$
  is a maximal subgroup isomorphic to~$G(\varphi)$. The argument
  in~\cite[Example~7.2]{Almeida:2005c} shows that $H$ cannot be
  relatively free with respect to any pseudovariety containing the
  two-element group. Hence, it suffices to show that the pseudovariety
  generated by the finite continuous homomorphic images of~$H$
  contains the two-element group, i.e., that $H$ has a continuous
  homomorphic image of finite even order. We claim, more specifically,
  that the alternating group $A_5$ is a continuous homomorphic image
  of~$H$.

  Let $A=\{a,b\}$ and let $h:\Om AS\to A_5$ be the unique
  continuous homomorphism such that $h(a)=(1\,2\,3)$ and
  $h(b)=(3\,4\,5)$. Note that $h$~is onto. To establish the
  claim, in view of Proposition~\ref{p:values-omega-powers-subst} it
  is enough to check that $h\circ\varphi^{12}|_A=h_A$.
  Although the length of the word $\varphi^n(a)$ depends exponentially
  on~$n$, the verification can be done easily by applying
  Lemma~\ref{l:technical} since
  $h\circ\varphi^{12}|_A=(\varphi_{A_5})^{12}(h|_A)$. The
  computation of $(\varphi_{A_5})^{12}(h|_A)$ can be carried
  out either by hand or by using a computer algebra system like
  GAP~\cite{GAP4:2006} and it confirms that indeed
  $(\varphi_{A_5})^{12}$ fixes $h|_A$, thereby proving the
  theorem.
\end{proof}

\subsection{The case of the Prouhet-Thue-Morse substitution}
\label{sec:finite-pres-prot}

Let $A$ be the two-letter alphabet $\{a,b\}$. The
\emph{Prouhet-Thue-Morse substitution} is the non-periodic primitive
substitution $\tau$ over $A$ given by $\tau(a)=ab$ and
$\tau(b)=ba$~\cite{Fogg:2002}. Note that no power of~$\tau$ is proper.
The word $aa$ is a connection for $\tau$ and $\tilde\tau=\tau^2$ is a
connective power of~$\tau$. The four elements of $X=X_\tau(a,a)$ are
$x=abba$, $y=ababba$, $z=abbaba$ and $t=ababbaba$, cf.~\cite[Section
3.2]{Balkova&Pelantova&Steiner:2008}. By Theorem~\ref{t:conjecture},
the \ci H-class $H=H_{aa}$ of $\tau^\omega(a)$, which is generated by
$\tot(X)$, admits the following presentation
\begin{equation}
  \label{eq:PTM-presentation-1}
  \langle X\mid \tilde\tau_{X,\pv G}^\omega(u)=u\ (u\in X)\rangle_\pv G.
\end{equation}
More precisely, the kernel of the continuous homomorphism $\Om XG\to
H$ that maps each $u\in X$ to $\tau^\omega(u)$ is the closed
congruence generated by the relations in the presentation
\eqref{eq:PTM-presentation-1}. Let $\alpha=\tot(x)$, $\beta=\tot(y)$,
$\gamma=\tot(z)$, and $\delta=\tot(t)$.

\begin{remark}\label{r:D-is-in-<A,B,C>}
  Let $\zeta$ be a continuous semigroup homomorphism from $\Om AS$
  into a profinite semigroup $S$. Suppose that $\zeta(x)$ belongs to a
  subgroup of~$S$. Then $\zeta(y)\cdot \zeta(x)^{\omega-1}\cdot
  \zeta(z)=\zeta(t)$.
\end{remark}

\begin{proof}
  We have
  $\zeta(ababba)\cdot \zeta(abba)^{\omega-1}\cdot \zeta(abbaba)
  =\zeta(ab)\cdot\zeta(abba)^{\omega+1}\cdot\zeta(ba)$,
  and $\zeta(abba)^{\omega+1}=\zeta(abba)$,
  because $\zeta(abba)$ is a group element of $S$.
\end{proof}

Applying Remark~\ref{r:D-is-in-<A,B,C>} to the continuous homomorphism
$\tot$, we conclude that $\beta\alpha^{-1}\gamma=\delta$ in $H$, so
that the profinite group $H$ is generated by $\{\alpha,\beta,\gamma\}$
and so the relation $yx^{-1}z=t$ turns out to be a consequence of the
relations in the presentation~\eqref{eq:PTM-presentation-1}.

A routine calculation shows that the images of letters of~$X$ by
$\tilde\tau_X$ are given by
$$
\tilde\tau_X(x)=zxy,\,
\tilde\tau_X(y)=ztxy,\,
\tilde\tau_X(z)=zxty,\,
\tilde\tau_X(t)=ztxty.$$
We proceed to give an alternative presentation of $G(\tau)$ as a
profinite group.

\begin{teor}\label{t:presentation-PTM}
  The group $G(\tau)$ admits the following presentation:
  \begin{equation}
    \label{eq:presentation-PTM}
    \langle x,y,z\mid 
    \Psi^\omega(x)=x,\,
    \Psi^\omega(y)=y,\,
    \Psi^\omega(z)=z
    \rangle_\pv G
  \end{equation}
  where $\Psi$ is the unique continuous endomorphism of
  \Om{\{x,y,z\}}G such that $\Psi(x)=zxy$, $\Psi(y)=zyx^{-1}zxy$, and
  $\Psi(z)=zxyx^{-1}zy$.
\end{teor}

\begin{proof}
  Let $X=\{x,y,z,t\}$ and $Y=X\setminus\{t\}$. Consider the continuous
  endomorphism $r$ of~\Om XG which fixes the elements of~$Y$ and maps
  $t$ to $yx^{-1}z$. Let $\Phi=\tilde\tau_{X,\pv G}$. Note that $\Psi$
  has been defined so that $r\circ\Phi$ coincides with $\Psi$ on $Y$.
  We extend $\Psi$
  to~\Om XG by putting $\Psi(t)=r(\Phi(t))$, which yields the equality
  $\Psi=r\circ\Phi$. On the other hand, we have
  $\Phi(t)=ztxty=ztxy\cdot (zxy)^{-1}\cdot zxty=\Phi(yx^{-1}z)$. As
  argued above, from Theorem~\ref{t:conjecture} it follows that
  $G(\tau)$ admits the presentation
  $$\langle X\mid yx^{-1}z=t,\, \Phi^\omega(u)=u\ (u\in X)\rangle_\pv G.$$
  To finish the proof, it now suffices to invoke
  Proposition~\ref{p:drop-relation}.
\end{proof}

For a profinite group $G$ and a pseudovariety of groups~\pv V, denote
by $G_\pv V$ the largest pro-\pv V factor group of~$G$.
For a prime $p$, let $\pv{Ab}_p$ denote the pseudovariety of all
elementary Abelian $p$-groups. The following result is well known
(cf.~\cite[Proposition~3.4.2 and Lemma~3.3.5]{Ribes&Zalesskii:2000}).

\begin{lema}
  \label{l:Abelianization-of relatively-free}
  Let \pv V and \pv W be pseudovarieties of groups such that $\pv
  V\subseteq\pv W$ and suppose that $G$ is a finitely generated free
  pro-\pv W group. Then $G_\pv V$ is a free pro-\pv V group and, if
  \pv V contains some nontrivial group, then the two groups have the
  same rank.
\end{lema}

In~\cite[Example~7.3]{Almeida:2005c} it was proved that the profinite
group $G(\tau)$ is not free on three generators: although the
computation starts from an incorrect set of return words, the same
argument goes through with the correct set. We may now adopt a
different approach to establish the following improvement.

\begin{teor}
  \label{t:not-relatively-free-profinite}
  The profinite group $G(\tau)$ is not relatively free.
\end{teor}

\begin{proof}
  We first note that, in view of Theorem~\ref{t:presentation-PTM}, the
  following presentation defines a finite quotient group $K_p$
  of~$G(\tau)$:
  \begin{align*}
    \langle x,y,z\mid\ 
    &
    \Psi^\omega(x)=x,\,
    \Psi^\omega(y)=y,\,
    \Psi^\omega(z)=z,\\
    &
    xy=yx,\,
    yz=zy,\,
    zx=xz,\,
    x^p=y^p=z^p=1
    \rangle_\pv G.
  \end{align*}
  Let $Y=\{x,y,z\}$. By Lemma~\ref{l:folklore}, the group $K_p$ also
  admits the presentation
  $$\langle x,y,z\mid
  \Lambda^\omega(x)=x,\,
  \Lambda^\omega(y)=y,\,
  \Lambda^\omega(z)=z
  \rangle_{\pv{Ab}_p},
  $$
  where $\Lambda$ is the continuous endomorphism of~$\Om Y{Ab}_p$
  induced by $\Psi$, which is given by $\Lambda(x)=zxy$ and
  $\Lambda(y)=\Lambda(z)=y^2z^2$. In the group $K_p$, we have
  $y=\Lambda^\omega(y)=\Lambda^\omega(z)=z$. Moreover, identifying each
  function $f$ from $Y$ to~$K_p$ with the triple $(f(x),f(y),f(z))$
  and applying iteratively the transformation $\Lambda_{K_p}\in\ci
  T(K_p^Y)$, one obtains inductively $\Lambda^n_{K_p}(x,y,y)
  =(xy^{2(4^n-1)/3},y^{4^n},y^{4^n})$. By Lemma~\ref{l:technical}, it
  follows that, in $K_p$ and for $n=m!$ sufficiently large, the
  equalities $xy^{2(4^n-1)/3}=x$ and $y^{4^n}=y$ hold. In particular,
  for $p=2$, we get $y=1$, which shows that $K_2$ is a cyclic group of
  order~2. On the other hand, for a prime $p>2$, the previous
  calculations show that $K_p$ is an elementary Abelian $p$-group of
  rank~two.
  
  Suppose that $G(\tau)$ were a free pro-\pv V group for some
  pseudovariety of groups~\pv V. By the above, \pv V contains
  $\pv{Ab}_p$ for every prime~$p$. Moreover, since
  $G(\tau)_{\pv{Ab}_p}=K_p$, the above calculations
  imply that $G(\tau)_{\pv{Ab}_p}$ has rank 1 for $p=2$ and rank 2 for
  an odd prime $p$, which contradicts Lemma~\ref{l:Abelianization-of
    relatively-free}. Hence, $G(\tau)$ cannot be a relatively free
  profinite group.
\end{proof}

The proof of Theorem~\ref{t:not-relatively-free-profinite} shows that
the rank of $G(\tau)$ is either two or three. The following result
settles the precise value of the rank. It is a further application of
the presentation of $G(\tau)$ given by Theorem \ref{t:presentation-PTM}.

\begin{teor}
  \label{t:rank}
  The group $G(\tau)$ has a group of order 18 of rank three as a
  continuous homomorphic image. Hence, $G(\tau)$ has rank three.
\end{teor}

\begin{proof}
  Set $Y=\{x,y,z\}$. Let $K$ be the group given by the
  following presentation
  $$\langle a,b,c\mid
  a^2=b^3=c^3=1,\, bc=cb,\, aba=b^2,\, aca=c^2\rangle_\pv G.
  $$
  Note that $K$ is the semidirect product of the subgroup $\langle
  b,c\rangle$, which is the direct product of two three-element
  groups, by the two-element subgroup~$\langle a\rangle$. Let
  $h:\Om YG\to K$ be the continuous homomorphism that sends
  $x,y,z$ respectively to~$a,b,c$.

  We first verify that $h(\Psi^2(u))=h(u)$ for all $u\in Y$.
  Since the calculations are quite similar, we treat only the case
  where $u=y$, leaving the other two cases for the reader to
  check:
  \begin{align*}
    h(\Psi^2(y))
    &=caba^{-1}cb\cdot cba^{-1}cab\cdot (cab)^{-1}\cdot caba^{-1}cb\cdot
    cab\cdot cba^{-1}cab\\
    &=c\cdot aba\cdot cbcbac\cdot aba\cdot cbcabcb\cdot aca\cdot b\\
    &=c\cdot b^2\cdot cbcbac\cdot b^2\cdot cbcabcb\cdot c^2\cdot b
    =b
    =h(y).
  \end{align*}
  From Proposition~\ref{p:values-omega-powers-subst}, it follows that
  $K$ is a continuous homomorphic image of~$G(\tau)$. Since it is
  easily checked that $K$ has rank three, it follows that so
  does~$G(\tau)$.
\end{proof}

One can also use the same technique as in the proof of
Theorem~\ref{t:eg:1} to establish that other finite groups are
continuous homomorphic images of~$G(\tau)$, as in the following
observation.

\begin{remark}
  \label{r:3}
  The alternating group $A_5$ is a continuous homomorphic image
  of~$G(\tau)$.
\end{remark}

\begin{proof}
  Let $A=\{a,b\}$ and consider the transformation $\tau_{A_5}\in
  A_5^A$ associated with the substitution $\tau$ according to
  Lemma~\ref{l:technical}. Identifying here $f\in A_5^A$ with the pair
  $(f(a),f(b))$, we have $\tau_{A_5}(x,y)=(xy,yx)$. Again, a
  straightforward calculation shows that $\tau^6_{A_5}$ fixes the
  pair of 3-cycles $((1\,2\,3),(3\,4\,5))$. Hence, for the continuous
  homomorphism $h:\Om AS\to A_5$ given by $h(a)=(1\,2\,3)$ and
  $h(b)=(3\,4\,5)$, by Lemma~\ref{l:technical} we obtain the
  equalities $h(\tau^\omega(a))=(1\,2\,3)$ and
  $h(\tau^\omega(a))=(3\,4\,5)$, from which it follows that
  $h(\tau^\omega(abba))=(1\,3\,2\,5\,4)$ and
  $h(\tau^\omega(ababba))=(1\,5\,2)$. This proves the claim since
  $A_5$~is generated by the latter two cycles, while
  $\tau^\omega(abba)$ and $\tau^\omega(ababba)$ belong to $G(\tau)$.
\end{proof}

We do not know whether the finite quotients of~$G(\tau)$ generate a
proper pseudovariety of groups. On the other hand, every finite cyclic
group is a continuous homomorphic image of~$G(\tau)$. Indeed, adding
to the defining relations of the presentation of~$G(\tau)$ given by
Theorem~\ref{t:presentation-PTM} the relations $y=z=1$, we obtain the
free procyclic group.

The following result adds further information about the presentation
of Theorem~\ref{t:presentation-PTM}.
  
\begin{prop}
  \label{p:non-trivial-relations-in-T}
  For each $u\in {\{x,y,z\}}$, the pseudoidentity
  $\Psi^\omega(u)=u$ fails in the two-element group $C_2$. Hence, for
  each $u\in {\{x,y,z\}}$, the relation
  $\Psi^\omega(u)=u$, which holds in~$G(\tau)$, is nontrivial.
\end{prop}

\begin{proof}
  Let $a$ be the nonidentity element of $C_2$ and let $Y=\{x,y,z\}$.
  Then, in the notation of Lemma~\ref{l:technical}, one verifies that
  the transformation $\Psi_{C_2}$ is idempotent. Moreover, if we
  identify each function $f\in C_2^Y$ with the triple
  $(f(x),f(y),f(z))$, then
  $\Psi_{C_2}(f_1,f_2,f_3)=(f_3f_1f_2,1,1)$. In particular, we
  obtain
  $\Psi^\omega_{C_2}(a,a,1)=\Psi^\omega_{C_2}(a,1,a)=(1,1,1)$.
  Hence, none of the pseudoidentities~$\Psi^\omega(u)=u$, with $u\in
  Y$, is satisfied by~$C_2$.
\end{proof}

\section{Open problems}
\label{sec:problems}

We end with a few open problems.

\begin{prob}
  \label{prob:0}
  Let $\varphi$ be a primitive substitution over a finite alphabet.
  And let $\pv V(\varphi)$ be the pseudovariety generated by the
  (continuous) homomorphic images of~$G(\varphi)$.
  \begin{enumerate}
  \item When is $\pv V(\varphi)=\pv G$?
  \item ``Compute'' $\pv V(\varphi)$.
  \end{enumerate}
\end{prob}

\begin{prob}
  \label{prob:1}
  \begin{enumerate}
  \item For which (minimal) subshifts $\ci X$, is the
    associated Sch\"utzenberger group $G(\ci X)$ decidable?
  \item In particular, is there any such group which is undecidable?
  \end{enumerate}
\end{prob}

It is well known that a free profinite group relatively to an
extension-closed pseudovariety is \pv G-projective as a profinite group
(cf.\ \cite{Gruenberg:1967} and~\cite[Corollary~11.2.3]{Wilson:1998})
and so, in view of the results of \cite{Almeida:2005c} or
\cite{Rhodes&Steinberg:2008}, finitely generated such groups certainly
appear as closed subgroups of free profinite semigroups on two
generators. P.~Zalesski\u{\i} asked in the \emph{Fields Workshop on
  Profinite Groups and Applications} (Carleton University, August
2005) and also in the Meeting of the \emph{ESI Programme on Profinite
  Groups} (Vienna, December 2008) whether in particular free pro-$p$
groups can appear as Sch\"utzenberger groups of free profinite
semigroups. In our setting, and in view of the results of this
section, this suggests the following question.

\begin{prob}
  \label{prob:2}
  Let $\Phi$ be a continuous endomorphism of\/~\Om XG and let $G$ be
  the profinite group presented by $\langle X\mid \Phi^\omega(x)=x\
  (x\in X)\rangle$.
  Under what assumptions on~$\Phi$ is $G$ a relatively free profinite
  group?
\end{prob}

As has been pointed out by the referee, if \pv H is a pseudovariety of
groups that contains every finite group whose Frattini quotient
belongs to~\pv H, then \Om XH admits such a presentation. Indeed every
free pro-\pv H group is \pv G-projective
\cite[Proposition~7.6.7]{Ribes&Zalesskii:2000}, and every finitely
generated \pv G-projective profinite group has such a presentation by
Corollary~\ref{c:there-is-some-presentation-of-that-kind}.

\subsubsection*{Acknowledgment}
We thank the anonymous referee for several suggestions, corrections,
and comments that contributed to a significant improvement in the
presentation of the paper.

\bibliographystyle{amsplain}

\begin{thebibliography}{10}

\bibitem{Almeida:1994a}
J.~Almeida, \emph{Finite semigroups and universal algebra}, World Scientific,
  Singapore, 1995, {E}nglish translation.

\bibitem{Almeida:1999c}
\bysame, \emph{Dynamics of implicit operations and tameness of pseudovarieties
  of groups}, Trans. Amer. Math. Soc. \textbf{354} (2002), 387--411.

\bibitem{Almeida:2005c}
\bysame, \emph{Profinite groups associated with weakly primitive
  substitutions}, Fundamentalnaya i Prikladnaya Matematika (Fundamental and
  Applied Mathematics) \textbf{11} (2005), no.~3, 13--48, In Russian. English
  version in J. Math. Sciences \textbf{144}, No. 2 (2007) 3881--3903.

\bibitem{Almeida:2003c}
\bysame, \emph{Profinite semigroups and applications}, Structural Theory of
  Automata, Semigroups, and Universal Algebra (New York) (Valery~B. Kudryavtsev
  and Ivo~G. Rosenberg, eds.), NATO Science Series II: Mathematics, Physics and
  Chemistry, vol. 207, Springer, 2005, Proceedings of the NATO Advanced Study
  Institute on Structural Theory of Automata, Semigroups and Universal Algebra,
  Montréal, Québec, Canada, 7-18 July 2003, pp.~1--45.

\bibitem{Almeida&ACosta:2007a}
J.~Almeida and A.~Costa, \emph{Infinite-vertex free profinite semigroupoids and
  symbolic dynamics}, J. Pure Appl. Algebra \textbf{213} (2009), 605--631.

\bibitem{Almeida&Volkov:2006}
J.~Almeida and M.~V. Volkov, \emph{Subword complexity of profinite words and
  subgroups of free profinite semigroups}, Int. J. Algebra Comput. \textbf{16}
  (2006), 221--258.

\bibitem{Balkova&Pelantova&Steiner:2008}
L.~Balkov{\'a}, E.~Pelantov{\'a}, and W.~Steiner, \emph{Sequences with constant
  number of return words}, Monatsh. Math. \textbf{155} (2008), no.~3-4,
  251--263.

\bibitem{ACosta:2006}
A.~Costa, \emph{Conjugacy invariants of subshifts: an approach from profinite
  semigroup theory}, Int. J. Algebra Comput. \textbf{16} (2006), 629--655.

\bibitem{ACosta&Steinberg:2011}
A.~Costa and B.~Steinberg, \emph{Profinite groups associated to sofic shifts
  are free.}, Proc. London Math. Soc. \textbf{102} (2011), 341--369.

\bibitem{Durand:1998}
F.~Durand, \emph{A characterization of substitutive sequences using return
  words}, Discrete Math. \textbf{179} (1998), no.~1-3, 89--101.

\bibitem{Durand:2000}
\bysame, \emph{Linearly recurrent subshifts have a finite number of
  non-periodic subshift factors}, Ergodic Theory Dynam. Systems \textbf{20}
  (2000), no.~4, 1061--1078.

\bibitem{Durand&Host&Skau:1999}
F.~Durand, B.~Host, and C.~Skau, \emph{Substitutional dynamical systems,
  {B}ratteli diagrams and dimension groups}, Ergodic Theory Dynam. Systems
  \textbf{19} (1999), no.~4, 953--993.

\bibitem{Fogg:2002}
N.~Pytheas Fogg, \emph{Substitutions in dynamics, arithmetics and
  combinatorics}, Lecture Notes in Mathematics, vol. 1794, Springer-Verlag,
  Berlin, 2002.

\bibitem{GAP4:2006}
The GAP~Group, \emph{{GAP -- Groups, Algorithms, and Programming, Version
  4.4}}, 2006, \verb+(http://www.gap-system.org)+.

\bibitem{Gruenberg:1967}
K.~W. Gruenberg, \emph{Projective profinite groups}, J. London Math. Soc.
  \textbf{42} (1967), 155--165.

\bibitem{Harju&Linna:1986}
T.~Harju and M.~Linna, \emph{On the periodicity of morphisms on free monoids},
  RAIRO Inf. Th\'eor. et Appl. \textbf{20} (1986), 47--54.

\bibitem{Hunter:1983}
R.~P. Hunter, \emph{Some remarks on subgroups defined by the {B}ohr
  compactification}, Semigroup Forum \textbf{26} (1983), 125--137.

\bibitem{Lind&Marcus:1996}
D.~Lind and B.~Marcus, \emph{An introduction to symbolic dynamics and coding},
  Cambridge University Press, Cambridge, 1996.

\bibitem{Lothaire:2001}
M.~Lothaire, \emph{Algebraic combinatorics on words}, Cambridge University
  Press, Cambridge, UK, 2002.

\bibitem{Lubotzky:2001}
A.~Lubotzky, \emph{Pro-finite presentations}, J. Algebra \textbf{242} (2001),
  672--690.

\bibitem{Margolis&Sapir&Weil:1995}
S.~Margolis, M.~Sapir, and P.~Weil, \emph{Irreducibility of certain
  pseudovarieties}, Comm. Algebra \textbf{26} (1998), 779--792.

\bibitem{Mosse:1992}
B.~Moss{\'e}, \emph{Puissances de mots et reconnaissabilit\'e des points fixes
  d'une substitution}, Theor. Comp. Sci. \textbf{99} (1992), no.~2, 327--334.

\bibitem{Mosse:1996}
\bysame, \emph{Reconnaissabilit\'e des substitutions et complexit\'e des suites
  automatiques}, Bull. Soc. Math. France \textbf{124} (1996), no.~2, 329--346.

\bibitem{Pansiot:1986}
J.-J. Pansiot, \emph{Decidability of periodicity for infinite words}, RAIRO
  Inf. Th\'eor. et Appl. \textbf{20} (1986), 43--46.

\bibitem{Rhodes&Steinberg:2008}
J.~Rhodes and B.~Steinberg, \emph{Closed subgroups of free profinite monoids
  are projective profinite groups}, Bull. London Math. Soc. \textbf{40} (2008),
  no.~3, 375--383.

\bibitem{Rhodes&Steinberg:2009qt}
\bysame, \emph{The $q$-theory of finite semigroups}, Springer Monographs in
  Mathematics, Springer, 2009.

\bibitem{Ribes&Zalesskii:2000}
L.~Ribes and P.~A. Zalesski{\u\i}, \emph{Profinite groups}, Ergeb. Math.
  Grenzgebiete 3, no.~40, Springer, Berlin, 2000.

\bibitem{Steinberg:2009}
B.~Steinberg, \emph{Maximal subgroups of the minimal ideal of a free profinite
  monoid are free}, Israel J. Math. \textbf{176} (2010), 139--155.

\bibitem{Wilson:1998}
J.~Wilson, \emph{Profinite groups}, London Mathematical Society Monographs, New
  Series, vol.~19, Clarendon, Oxford, 1998.

\end{thebibliography}

\end{document}